\documentclass[12pt,reqno]{amsart}

\usepackage{amssymb,amsmath,amsthm}
\usepackage{mathrsfs}
\usepackage[margin=30mm]{geometry}
\usepackage{xcolor}
\usepackage{mathtools}
\mathtoolsset{showonlyrefs}

\usepackage{here}

\theoremstyle{definition}
\newtheorem{thm}{Theorem}[section]
\newtheorem{pro}[thm]{Proposition}
\newtheorem{lem}[thm]{Lemma}
\newtheorem{cor}[thm]{Corollary}

\theoremstyle{definition}
\newtheorem{dfn}{Definition}

\newtheorem{rem}{Remark}

\newtheorem{notation}{Notation}
\newtheorem*{conjecture}{Conjecture}

\usepackage{ascmac}

\newcommand{\bs}[1]{{\boldsymbol #1}}
\newcommand{\INN}[2]{\langle #1,#2 \rangle}
\newcommand{\tr}{\operatorname{tr}}

\newcommand{\Ker}{\operatorname{Ker}}
\newcommand{\id}{\operatorname{id}}

\usepackage[%
 setpagesize=false,%
 bookmarks=true,%
 bookmarksdepth=tocdepth,%
 bookmarksnumbered=true,%
 colorlinks=true,%
]{hyperref}

\title[Equifocal hypersurfaces and backward mean curvature flows]{Equifocal hypersurfaces in symmetric spaces of compact type and backward mean curvature flows}

\author{Kurando Baba}
\address{(K.B.) Department of Mathematics,
Faculty of Science and Technology,
Tokyo University of Science,
2641 Yamazaki, Noda, 278-8510, Chiba, Japan\\
Research Institute for Science and Technology at Tokyo University of Science,
Division of Joint Research of Geometry and Natural Science\\
Institut f\"ur Mathematik,
Universit\"at Augsburg, 86135 Augsburg, Germany
}
\email[primary]{kurando.baba@rs.tus.ac.jp}
\email[secondary]{kurando.baba@uni-a.de}
\thanks{The first author was partially supported by JSPS KAKENHI Grant Number 25K07017.}

\author{Naoyuki Koike}
\address{(N.K.) Department of Mathematics,
Faculty of Science,
Tokyo University of Science,
1-3 Kagurazaka, Shinjuku-ku, 162-8601, Tokyo, Japan\\
Research Institute for Science and Technology at Tokyo University of Science,
Division of Joint Research of Geometry and Natural Science}
\email{koike@rs.tus.ac.jp}
\thanks{The second author was partially supported by JSPS KAKENHI Grant Number 26K06815.}

\date{\today}



\begin{document}
\maketitle

\begin{abstract}
We first derive a formula for the mean curvature
and the squared norm of the shape operator
of equifocal hypersurfaces in simply-connected irreducible symmetric spaces of compact type.
The formulas are given explicitly in terms of the tangential focal data of the equifocal hypersurfaces.
Third, we study the backward mean curvature flow for equifocal hypersurfaces.
The long-time existence of this flow for an equifocal hypersurface was established by Liu and Radeschi.
We analyze
the time evolution of the mean curvature and the squared norm of the shape operator along the long-time solution,
thereby we generalize the result of Liu and Terng for isoparametric hypersurfaces in the sphere.
Our analysis also gives an extension of Liu-Terng conjecture on the backward mean curvature flows
in the sphere to the simply-connected irreducible symmetric space of compact type.
\end{abstract}

\section{Introduction}

The mean curvature flow (abbreviated as MCF) is a geometric flow that evolves
submanifolds in the direction of their mean curvature vector fields,
and is expressed as the solution of a certain type of weak parabolic partial differential equation
(which is called the mean curvature flow equation).
The study of the MCF for certain kind of submanifolds has proceeded
along two parallel tracks: the analysis of forward singularities
and the analysis of backward ancient solutions (if there exist).
In both tracks, for a general initial submanifold,
neither a complete description of the long-time behavior nor an explicit expression of the solution
is available;
complete descriptions
and explicit expression are available only when the initial submanifold carries high symmetry.
The reason is that, for an initial submanifold with high symmetry,
the mean curvature flow equation reduces to an ordinary differential equation (abbreviated as ODE) on the normal slice.
As a typical case, Liu-Terng \cite{LiuTe09, LiuTe20} performed this reduction in both directions
for the case when the initial submanifold is an isoparametric submanifold in Euclidean spaces or in spheres.
Related developments have also been obtained in non-compact ambient spaces.
In the forward direction, the second author \cite{Koike14} studied the MCF for certain curvature-adapted isoparametric submanifolds
and their focal submanifolds in symmetric spaces of non-compact type.
Liu-Yang \cite{LiuWa26} investigated the MCF in both directions for isoparametric submanifolds
in hyperbolic spaces.

The purpose of this paper is to develop an analogue of Liu-Terng theory
for symmetric spaces of compact type with arbitrary rank.
In this setting, the natural class of submanifolds extending isoparametric hypersurfaces in the sphere is the class of equifocal submanifolds in the sense of Terng-Thorbergsson \cite{TeTh95}.
An equifocal submanifold is a class of submanifolds that 
inherits properties
characterizing isoparametric hypersurfaces in the sphere
within the framework of symmetric spaces of compact type.
Namely, equifocal submanifolds
have the global flatness of the normal bundle and the invariance
of a focal structure,
which is defined by the focal radius along the normal direction,
under parallel translations with respect to the normal connection
(see Definition \ref{dfn:def_equifocal}).
The theory of the MCF in the forward direction
was established by the second author \cite{Koike11}
by extending the results of Liu-Terng in \cite{LiuTe09} to the framework of equifocal submanifolds.
In fact, he constructed the solution of the MCF in the forward direction
and proved that the solution preserves the equifocality.
Furthermore, 
the solution collapse in finite time to a focal submanifold if its initial submanifold is not minimal.
The backward direction has been explored only partially by Liu-Radeschi \cite{LiuRa22}.
They showed that the flow in the backward direction has a long-time solution
and converges to a unique minimal equifocal submanifold.
Their method analyzes the flow without providing explicit solutions, unlike in \cite{LiuTe20}.
However, the asymptotic analysis of long-time solutions in the backward direction, carried out in \cite{LiuTe20} for isoparametric hypersurfaces, is still missing in symmetric spaces of compact type in a general setting.

In this paper,
we first derive formulas for the mean curvature
and the squared (Hilbert-Schmidt) norm of the shape operators
of equifocal hypersurfaces in a simply-connected irreducible symmetric space $N$
of compact type in terms of the geometric data encoded in their focal structure.
Assume that the Einstein constant of $N=G/K$ (which is Einstein) is equal to $\kappa$.
These are formulated by using the result of Terng-Thorbergsson (cf.~\cite[Theorem 1.6]{TeTh95}).
From this result we use the tangential focal data
$\Gamma^{F}(M)=\{\theta_{F},l/(2g_{F}),\{m^{F}_{1},m^{F}_{2}\};2g_{F}\}$.
See the statement of Theorem \ref{thm_TeTh95_thm1.6.b}
about this data.
With this notation, the formulas are stated as follows.

\begin{thm}[Theorem \ref{thm:equifocal_HA2_formula}]\label{thm:main1}
Let $M$ be an equifocal hypersurface of $N$.
We set constants $\delta_{F}$, $\bar{\theta}_{F}$
by
\begin{equation}
\delta_{F}:=\dfrac{m_{1}^{F}-m_{2}^{F}}{m_{1}^{F}+m_{2}^{F}},
\quad
\bar{\theta}_{F}:=\pi\dfrac{2g_{F}}{l}\cdot \theta_{F}.
\end{equation}
\begin{enumerate}
\item[(i)] The mean curvature $H$ of $M$ is expressed as
\begin{equation}
H=\dfrac{\pi g_{F}(m^{F}_{1}+m^{F}_{2})}{l \cdot \sin\bar{\theta}_{F}}
(\delta_{F}+\cos\bar{\theta}_{F}).
\end{equation}
\item[(ii)] The squared norm $\|A\|^{2}$ of the shape operator of $M$
is expressed as
\begin{equation}
\|A\|^{2}
=\dfrac{1}{2\epsilon_{F}\sin^{2}\bar{\theta}_{F}}(1+\delta_{F}\cos\bar{\theta}_{F})-\kappa,
\end{equation}
where $\epsilon_{F}$ is the constant given by
\begin{equation}
\epsilon_{F}:=\dfrac{1}{\pi^{2}}\left(\dfrac{l}{2g_{F}}\right)^{2}\cdot\dfrac{1}{m^{F}_{1}+m^{F}_{2}}.
\end{equation}
\end{enumerate}
\end{thm}

Theorem \ref{thm:main1} extends the result of Liu-Terng
\cite[(4.4) in Lemma 4.1 and (4.15) in Lemma 4.7]{LiuTe20} for isoparametric hypersurfaces in spheres.
The essential point is that, whereas the mean curvature $H$
of an isoparametric hypersurface in a sphere are described by the principal curvature data, in the equifocal case they are determined only by the tangential focal data $\Gamma^{F}(M)$.

\begin{rem}\label{rem:ricxixi}
If $N$ is reducible,
then it is not Einstein.
In this case, we can show that,
for a globally unit normal vector field $\bs{\xi}$ of the equifocal hypersurface $M$,
$\operatorname{Ric}^{N}(\bs{\xi},\bs{\xi})$ is constant on $M$ (see Appendix).
Then as $\kappa:=\operatorname{Ric}^{N}(\bs{\xi},\bs{\xi})$,
the statement of Theorem \ref{thm:main1} holds.
\end{rem}

The second aim of this paper is to provide explicit descriptions of
the mean curvature and
the squared norm of the shape operator along the MCF in the backward direction starting from
equifocal hypersurfaces $M$ in $N$,
and to give the asymptotic analysis of the long-time solutions performed in \cite{LiuTe20}.
The MCF starting from $M$ reduces to the ODE \eqref{eqn:MCF_ODE},
which is described by a vector field defined on a fundamental domain of the affine Coxeter group associated with $M$ (see \cite[p.~106]{Koike11}).
It follows from \cite[Lemma 3.1]{Koike11} that the solution to the mean curvature flow starting from
$M$ is given by the family of parallel hypersurfaces determined by the integral curves of this ODE.
Due to the second author
(\cite{Koike11})
and
Liu-Radeschi
(\cite{LiuRa22}),
the maximal solution
$\{M_{t}\}_{t}$
($M_{0}=M$)
of the MCF starting from an equifocal hypersurface $M$
exist in the range
$-\infty<t<T$,
where $T$ is a positive real number if $M$ is not minimal.
Let $A(t)$ and $H(t)$ denote
the shape operator and the mean curvature of $M_{t}$, respectively.

First, we describe explicitly $\|A(t)\|^{2}$ and 
the exponential decay of the mean curvature $H(t)$ as $t\to -\infty$.
It follows from Theorem \ref{thm:main1}
that the minimal equifocal hypersurface where $\{M_{t}\}_{t}$ converges as $t \to -\infty$
is characterized as a solution to the equation $\delta_{F} + \cos\bar{\theta}_{F} = 0$.
By using this fact we obtain the following theorem.

\begin{thm}[{Propositions \ref{pro:FRA_Ht_-inf} and \ref{pro:FRA_At_-inf}}]\label{thm:BMCF_evaluate}
Assume that $M_{0}=M$ is not minimal.
Under the above settings we have the followings:
\begin{enumerate}
\item[(i)] The following equality holds:
\begin{equation}
\lim_{t\to-\infty}
H(t)^{2}\exp(-\dfrac{1}{\epsilon_{F}}\cdot t)
=\dfrac{m^{F}_{1}+m^{F}_{2}}{4\epsilon_{F}(1-\delta_{F}^{2})}
(\delta_{F}+\cos\bar{\theta}_{F})^{2}.
\end{equation}
\item[(ii)] The following equality holds:
\begin{equation}
\displaystyle{\lim_{t\to-\infty}\|A(t)\|^{2}
=\dfrac{1-2\kappa\epsilon_{F}}{2\epsilon_{F}}
}.
\end{equation}
\end{enumerate}
\end{thm}

Theorem \ref{thm:BMCF_evaluate}
extends
the results of Liu-Terng
\cite[Theorem 4.14]{LiuTe20}.
Here, we note that $0<\epsilon_{F}\leq 1$ holds (see \eqref{eqn:epF_range} below).

Second,
by using Theorem \ref{thm:main1},
we give an explicit
estimate of the ratio $\|A(t)\|^{2}/H(t)^{2}$ along the solution $\{M_{t}\}_{t}$ as follows.

\begin{thm}[{Theorem \ref{thm:FRA_At_Ht_ratio}}]\label{thm:main3}
Assume that
$M_{0}=M$ is not minimal.
Then we obtain the following assertions:
\begin{enumerate}
\item[(i)] In the case where $\delta_{F}=0$ and $\epsilon_{F}=\frac{1}{2\kappa}$ holds,
we have
\begin{equation}
\dfrac{\|A(t)\|^{2}}{H(t)^{2}}\equiv
\dfrac{2}{m_{1}^{F}+m_{2}^{F}}.
\end{equation}
\item[(ii)] Otherwise,
the following inequality holds:
\begin{equation}
\dfrac{\|A(t)\|^{2}}{H(t)^{2}}
<\dfrac{4\max\{m_{1}^{F},m_{2}^{F}\}}{(m^{F}_{1}+m^{F}_{2})^{2}}
\cdot\dfrac{1}{(\delta_{F}+\cos\bar{\theta}_{F})^{2}}\cdot
\exp(-\dfrac{1}{\epsilon_{F}}\cdot t).
\end{equation}
\end{enumerate}
\end{thm}

Theorem \ref{thm:main3} extends the result of Liu-Terng
\cite[Corollary 4.12]{LiuTe20} for isoparametric hypersurfaces in spheres.

Third, we give two distinct aspects that are expected to contribute to the study of the characterization
of equifocal hypersurfaces in terms of MCF theory.

\begin{thm}[{Theorem \ref{thm:BMCF_HA_evaluate}}]\label{thm:BMCF_evaluate2}
Assume that
$M_{0}=M$ is not minimal.
In the case of $\epsilon_{F} < \frac{1}{2\kappa}$,
there exist $t_{0}>0$ and $c_{1},c_{2}>0$ satisfying the following inequalities:
\begin{equation}\label{eqn:conj1}
c_{2}\cdot\exp(-\dfrac{1}{\epsilon_{F}}\cdot t)
\leq 
\dfrac{\|A(t)\|^{2}}{H(t)^{2}}
\leq c_{1}\cdot\exp(-\dfrac{1}{\epsilon_{F}}\cdot t)
\quad
(t<-t_{0}).
\end{equation}
\end{thm}

We also derive an estimate for the traceless part 
$\mathring{A}(t)$ of $A(t)$
along the solution $\{M_{t}\}_{t}$.
Let $\phi(t)$
denote the squared norm of $\mathring{A}(t)$,
that is,
\begin{equation}
\phi(t):=\|\mathring{A}(t)\|^{2},
\quad
\mathring{A}(t):=A(t)-\dfrac{H(t)}{\dim M_{t}}\mathrm{id}.
\end{equation}
Then we have the following theorem.

\begin{thm}[Theorem \ref{thm:prop_phi}]\label{thm:BMCF_evaluate3}
For any $0<h<1$, there exists $c_{0}>0$ such that
one of the following assertions holds:
\begin{enumerate}
\item[(i)] The case of $\delta_{F}=0$:
If $\theta_{F}\in (\theta_{\min}-c_{0},\theta_{\min}+c_{0})$, we have
\begin{equation}\label{eqn:cor_phi1}
\dfrac{1-2\kappa\epsilon_{F}}{2\epsilon_{F}}
\leq
\phi(t)
\leq
\dfrac{1-2\kappa\epsilon_{F}}{2\epsilon_{F}}+h\quad
(t<0).
\end{equation}
\item[(ii)] The case of $\delta_{F}>0$: If $\theta_{F}\in (\theta_{\min}-c_{0},\theta_{\min})$,
we have
\begin{equation}\label{eqn:cor_phi2}\noeqref{eqn:cor_phi2}
\dfrac{1-2\kappa\epsilon_{F}}{2\epsilon_{F}}-h
\leq
\phi(t)
\leq
\dfrac{1-2\kappa\epsilon_{F}}{2\epsilon_{F}}
\quad(t<0).
\end{equation}
If $\theta_{F}\in (\theta_{\min},\theta_{\min}+c_{0})$, we have
\begin{equation}\label{eqn:cor_phi3}\noeqref{eqn:cor_phi3}
\dfrac{1-2\kappa\epsilon_{F}}{2\epsilon_{F}}\leq
\phi(t)
\leq
\dfrac{1-2\kappa\epsilon_{F}}{2\epsilon_{F}}+h
\quad
(t<0).
\end{equation}
\item[(iii)] In the case of $\delta_{F}<0$:
If $\theta_{F}\in (\theta_{\min}-c_{0},\theta_{\min})$,
we have
\begin{equation}\label{eqn:cor_phi4}\noeqref{eqn:cor_phi4}
\dfrac{1-2\kappa\epsilon_{F}}{2\epsilon_{F}}
\leq
\phi(t)
\leq
\dfrac{1-2\kappa\epsilon_{F}}{2\epsilon_{F}}+h
\quad
(t<0).
\end{equation}
If $\theta_{F}\in (\theta_{\min},\theta_{\min}+c_{0})$,
we have
\begin{equation}\label{eqn:cor_phi5}
\dfrac{1-2\kappa\epsilon_{F}}{2\epsilon_{F}}-h
\leq
\phi(t)
\leq
\dfrac{1-2\kappa\epsilon_{F}}{2\epsilon_{F}}
\quad(t<0).
\end{equation}
\end{enumerate}
\end{thm}

Theorems
\ref{thm:BMCF_evaluate2} and
\ref{thm:BMCF_evaluate3} extend
the results of Liu-Terng
\cite[Corollary 4.15 and Proposition 4.17]{LiuTe20}.
Theorem \ref{thm:BMCF_evaluate2} describes how the balance of curvature changes over time, while
Theorem \ref{thm:BMCF_evaluate3} gives restrictions on the range of possible curvature values as rigidity conditions.
These results are expected to provide deeper insight into the classification of ancient solutions of MCF.
Namely,
we can generalize Liu-Terng conjectures
\cite[Conjectures 4.16, 4.18]{LiuTe20}
for the MCFs in the sphere
to those in the simply-connected irreducible symmetric space of compact type
as follows:

\begin{conjecture}[generalized Liu-Terng conjecture]\label{conj:gLTc}
Let $M$ be a compact hypersurface in a simply-connected irreducible symmetric space $N$ of compact type.
Let $\bs{\xi}$ be a unit normal vector field on $M$.
Assume that
the Einstein constant of $N$ is equal to $\kappa$.
Let $\{M_{t}\}_{t\in(-T_{\min},T_{\max})}$
denote the maximal solution of the MCF
starting from $M$,
and $A(t)$ and $H(t)$ denote the shape operator
and the mean curvature of $M_{t}$, respectively.
Assume that $T_{\min}=\infty$.
\begin{enumerate}
\item[(i)] If $A(t)$ and $H(t)$ satisfy the inequality \eqref{eqn:conj1}
for some $\epsilon\in(0,1/(2\kappa))$ (instead of $\epsilon_{F}$),
then 
$M$ is an equifocal hypersurface satisfying $\epsilon_{F}=\epsilon$,
where $\epsilon_{F}$ is as in the tangential focal data of the equifocal
hypersurface.
\item[(ii)] If $A(t)$ and $H(t)$
satisfy one of the inequalities
\eqref{eqn:cor_phi1}--\eqref{eqn:cor_phi5}
for some $\epsilon\in(0,1/(2\kappa))$ (instead of $\epsilon_{F}$)
and $h\in(0,1)$,
then 
$M$ is an equifocal hypersurface satisfying $\epsilon_{F}=\epsilon$.
\end{enumerate}
\end{conjecture}

The original Liu-Terng conjecture was discussed in \cite{LiuTe20}.
In particular, they explained its connection to Chern conjecture in the static case,
which is known as one of the problems in the problem section due to Yau \cite{Yau82}.
Our results of this paper may suggest a possible extension of Chern conjecture to minimal hypersurfaces in irreducible symmetric spaces of compact type.

This paper is organized as follows.
In Section \ref{sec:pre}, we recall some fundamental results on equifocal hypersurfaces
that will be used throughout the paper.
In Section \ref{sec:FRA}, we prove Theorems \ref{thm:main1}.
In Section \ref{sec:FRA_MCF}, we analyze the long-time solutions of the backward MCF
starting from an equifocal hypersurface
and prove the remaining results stated in the Introduction.
In Appendix \ref{sec:const_ricxixi},
we give a proof of the constancy of $\operatorname{Ric}^N(\xi,\xi)$
stated in Remark \ref{rem:ricxixi}.

\section{Preliminaries}\label{sec:pre}

Let $N$ be a complete Riemannian manifold
with Riemannian curvature tensor field $R$
and exponential map $\exp$.
Let $M$ be an immersed submanifold in $N$.
Denote by 
$\psi:T^{\perp}M\to M$ the normal bundle of $M$
and by $\exp^{\perp}$ the normal exponential map
of $M$ which is defined by
\begin{equation}
\exp^{\perp}(\bs{\xi})=\exp_{x}(\bs{\xi})
\quad
(x\in M,\,
\bs{\xi}\in T^{\perp}_{x}M).
\end{equation}
A normal vector $\bs{\xi}\in T^{\perp}_{x}M$
($x\in M$)
is said to be \textit{singular} for $\exp^{\perp}$
if $d\psi_{\bs{\xi}}(\operatorname{Ker}d\exp^{\perp}_{\bs{\xi}})\neq\{0\}$
holds,
in which case 
$\bs{\xi}$ is called a \textit{focal normal}
at $x$
and the point $\exp^{\perp}(\bs{\xi})\in N$
is called a \textit{focal point}
of $M$ at $x$.
For any focal normal $\bs{\xi}$,
the integer
$m^{F}_{\bs{\xi}}:=\dim_{\mathbb{R}}
d\psi_{\bs{\xi}}(\operatorname{Ker}d\exp^{\perp}_{\bs{\xi}})$
is called its \textit{multiplicity},
and its length $\|\bs{\xi}\|$
is called its \textit{focal radius}.
The set of pairs consisting of focal normals
at $x$ and their multiplicities
are called the \textit{tangential focal data}
of $M$ at $x$, which we write $\Gamma^{F}(M,x)$,
and $\Gamma^{F}(M)=\bigcup_{x\in M}\Gamma^{F}(M,x)$
is called the tangential focal data of $M$.

\subsection{Equifocal hypersurfaces and their tangential focal data}

We recall the notion of equifocal submanifolds
in symmetric spaces of compact type,
which was introduced by Terng-Thorbergsson \cite{TeTh95}.

\begin{dfn}[{\cite{TeTh95}}]\label{dfn:def_equifocal}
A compact immersed submanifold $M$ in a symmetric space
$N$ of compact type
is said to be \textit{equifocal}
if it satisfies the following conditions:
\begin{enumerate}
\item[(i)] $T^{\perp}M$ is globally flat, that is,
the normal connection $\nabla^{\perp}$ is flat and has trivial holonomy group.
\item[(ii)] $T^{\perp}M$ is abelian,
that is,
for each $x\in M$,
the set
$\exp_{x}(T^{\perp}_{x}M)$
is contained in a flat totally geodesic submanifold of $N$.
\item[(iii)] $\Gamma^{F}(M)$ is invariant under the parallel translation with respect to $\nabla^{\perp}$.
\end{enumerate}
Here,
we note that
$\exp_{x}(T^{\perp}_{x}M)$ is called a \textit{flat section} of $M$.
\end{dfn}

The notion of an equifocal submanifold is modeled on two fundamental classes of submanifolds: isoparametric hypersurfaces in spheres
and principal orbits of hyperpolar actions on symmetric spaces of compact type.
In this paper,
we focus our attention on equifocal hypersurfaces
in symmetric spaces of compact type
as the main objects of study.

We first recall some results about equifocal hypersurfaces
following \cite{TeTh95}.
Let $N$ be a simply-connected symmetric space of compact type.
Throughout this paper, we always assume that $N$ is simply-connected.
Let $M$ be an equifocal hypersurface in $N$.
It follows from \cite[Theorem 1.6, (c)]{TeTh95}
that $M$ is embedded.
Let $\bs{\xi}$ be a unit normal vector field on $M$.
For a real number $t$,
the map $\eta_{t\bs{\xi}}:M\to N$
defined by $\eta_{t\bs{\xi}}(x):=\exp^{\perp}(t\bs{\xi}_{x})$
for $x\in M$
is called the endpoint map of $M$ for $t\bs{\xi}$.
According to (iii) of Definition \ref{dfn:def_equifocal},
$\eta_{t\bs{\xi}}$ is of constant rank.
The image $M_{t\bs{\xi}}:=\eta_{t\bs{\xi}}(M)$ of $\eta_{t\bs{\xi}}$
is called a \textit{parallel submanifold} of $M$
if $\eta_{t\bs{\xi}}$ is of full rank;
otherwise,
$M_{t\bs{\xi}}$
is called a \textit{focal submanifold} of $M$.
Any parallel submanifold of $M$
is diffeomorphic to $M$ via $\eta_{t\bs{\xi}}$.
Fix $x\in M$.
We write $\gamma_{\bs{\xi}_{x}}$
as the normal geodesic of $M$
satisfying $\gamma_{\bs{\xi}_{x}}(0)=x$
and $\gamma_{\bs{\xi}_{x}}'(0)=\bs{\xi}_{x}$.
Then it follows from 
\cite[Theorem 1.6, (a)]{TeTh95}
that $T_{x}:=\mathrm{Im}(\gamma_{\bs{\xi}_{x}})$
is a circle in $N$,
which is called the \textit{normal circle} of $M$ at $x$.
Furthermore,
the length of $T_{x}$ does not depend on the choice of $x$,
which is denoted by $l$.
Hence the endpoint map $\eta_{l\cdot\bs{\xi}}:M\to N$
gives the identity transformation on $M$.

The following theorem is also due to Terng-Thorbergsson \cite{TeTh95}.

\begin{thm}[{cf. \cite[Theorem 1.6, (b)]{TeTh95}}]\label{thm_TeTh95_thm1.6.b}
There exist a real number $\theta_{F}$
and integers $2g_{F}$, $m^{F}_{1}$,
$m^{F}_{2}$ satisfying the following conditions:
\begin{enumerate}
\item[(i)] $0<\theta_{F}<\dfrac{l}{2g_{F}}$.
\item[(ii)] The set $\mathcal{F}_{x}$
of focal normals of $M$ at $x$ is expressed as follows:
\begin{equation}
\mathcal{F}_{x}=\{t_{j,k}\cdot\bs{\xi}_{x}\mid
j\in\{1,\dotsc,2g_{F}\},k\in\mathbb{Z}
\},
\end{equation}
where $t_{j,k}\in\mathbb{R}$ is defined by
\begin{equation}\label{eqn:tjk_dfn}
t_{j,k}:=\theta_{F}+(j-1)\dfrac{l}{2g_{F}}+l\cdot k.
\end{equation}
Furthermore,
for each $j\in\{1,\dotsc,2g_{F}\}$
and $k\in\mathbb{Z}$,
the multiplicity of 
$t_{j,k}\cdot\bs{\xi}_{x}$ satisfies the following relation:
\begin{equation}\label{eqn:multF_inv}
m^{F}_{t_{j,k}\cdot\bs{\xi}_{x}}=m^{F}_{\text{$j$ ($\mathop{\mathrm{mod}}2$)}},
\end{equation}
where $m^{F}_{\text{$j$ ($\mathop{\mathrm{mod}}2$)}}$
means $m^{F}_{1}$ for odd $j$
and $m_{2}^{F}$ for even $j$.
\item[(iii)] The group generated by affine reflections 
with respect to focal normals at $x$
on $T^{\perp}_{x}M$
is isomorphic to the infinite dihedral group.
\end{enumerate}
\end{thm}

The group stated in Theorem \ref{thm_TeTh95_thm1.6.b}, (3)
is called the \textit{affine Coxeter group} of $M$ at $x$,
which we write $W_{x}$.
It follows from Theorem \ref{thm_TeTh95_thm1.6.b}, (2)
that $W_{x}$ preserves $\mathcal{F}_{x}$ invariantly.
Then the orbit space
$\mathcal{F}_{x}/W_{x}$ is given by $\{t_{1,0}\cdot\bs{\xi}_{x}(=\theta_{F}\cdot\bs{\xi}_{x}), t_{2,0}\cdot\bs{\xi}_{x}\}$.
In other words,
\eqref{eqn:multF_inv}
asserts that
the multiplicity $m^{F}_{\bs{v}}$
($\bs{v}\in\mathcal{F}_{x}$) is invariant under the action of $W_{x}$.
By taking these facts into account,
we define the \emph{tangential focal data} $\Gamma^{F}(M,x)$ of $M$ at $x$ by
\begin{equation}\label{dfn:tangentialfocaldata}
\Gamma^{F}(M)=\Gamma^{F}(M,x)=\left\{\theta_{F},\dfrac{l}{2g_{F}},\{m^{F}_{1},m^{F}_{2}\};2g_{F}\right\}.
\end{equation}

\begin{rem}\label{rem:sphere_F_A}
Let us consider the case when $M$ is an isoparametric hypersurface in the $(n+1)$-sphere $S^{n+1}$.
Then the tangential focal data of $M$ coincides with the principal curvature data.
Indeed, $g_{F}$ is equal to the number of distinct principal curvatures, and
$m_{1}^{F}$ and $m_{2}^{F}$ are precisely the two possible multiplicities of the principal curvatures.
\end{rem}

\subsection{Parallel hypersurfaces}\label{sec:parallel_FRA}

Let $N$ be a simply-connected irreducible symmetric space of compact type.
Let $M$ be an equifocal hypersurface in $N$
with unit normal vector field $\bs{\xi}$.
The following is known by \cite[Theorem 1.6, (e, f, g)]{TeTh95}:
\begin{itemize}
\item $M_{t\cdot\bs{\xi}}$
is an equifocal hypersurface
for $t\in(\theta_{F}-l/(2g_{F}),\theta_{F})$.
\item $\{M_{t\cdot\bs{\xi}}\mid
t\in[\theta_{F}-l/(2g_{F}),\theta_{F}]\}$
gives a singular Riemannian foliation on $G/K$.
\end{itemize}
Fix $x\in M$
and let $\gamma_{\bs{\xi}_{x}}$
denote the normal geodesic of $M$
satisfying 
$\gamma_{\bs{\xi}_{x}}(0)=x$
and $\gamma_{\bs{\xi}_{x}}'(0)=\bs{\xi}_{x}$.
We define the subset $\widetilde{C}$ of $T^{\perp}_{x}M$ as follows:
\begin{equation}\label{eqn:Coxeter_domain}
\widetilde{C}:=\left\{t\cdot\bs{\xi}_{x}\,\middle|\, \theta_{F}-\dfrac{l}{2g_{F}}<t<\theta_{F}\right\},
\end{equation}
which is a fundamental domain
of the affine Coxeter group
of $M$ at $x$.
Fix $t\in (\theta_{F}-l/(2g),\theta_{F})$.
Let $y:=\eta_{t\bs{\xi}}(x)$ be a point of $M_{t\bs{\xi}}$.
Then $\gamma_{\bs{\xi}_{x}}'(t)$
is a unit normal vector of $M_{t\bs{\xi}}$
at $y$.
Let $\bs{\xi}^{t}$ be the parallel normal vector field
of $M_{t\bs{\xi}}$ with $\bs{\xi}^{t}_{x}=\gamma_{\bs{\xi}_{x}}'(t)$.
Clearly, we have the following lemma.

\begin{lem}
Let $t\in(\theta_{F}-l/(2g_{F}),\theta_{F})$.
The tangential focal data
$\Gamma^{F}(M_{t\bs{\xi}})$
is expressed as follows:
\begin{equation}
\Gamma^{F}(M_{t\bs{\xi}})
=\left\{
\theta_{F}-t,\dfrac{l}{2g_{F}},\{m^{F}_{1},m^{F}_{2}\};2g_{F}\right\}.
\end{equation}
\end{lem}

\section{Liu-Terng-type formulas for equifocal hypersurfaces}\label{sec:FRA}

The formulas for the mean curvature and the squared norm of the shape operator of an isoparametric hypersurface in a sphere were obtained by Liu-Terng \cite[(4.4) and (4.15)]{LiuTe20}.
In this section, we first extend their formulas to the mean curvature and the squared norm of the shape operator of
an equifocal hypersurface
in a simply-connected irreducible symmetric space of compact type (Theorem \ref{thm:equifocal_HA2_formula}).
Our formulas are described in terms of the tangential focal data of the hypersurfaces.
Next, we apply these formulas to their parallel hypersurfaces.

\subsection{Liu-Terng type formulas for equifocal hypersurfaces}

Let $N$ be a simply-connected irreducible symmetric space of compact type
and $\kappa$ be the Einstein constant of $N$.
Let $M$ be an equifocal hypersurface in $N$ with tangential focal data
$\Gamma^{F}(M)=\{\theta_{F},l/2g_{F},\{m_{1}^{F},m_{2}^{F}\};2g_{F}\}$.

\begin{notation}
We define the three constants $\delta_{F}$, $\bar{\theta}_{F}$ and $\epsilon_{F}$ as follows:
\begin{equation}
\delta_{F}:=\dfrac{m_{1}^{F}-m_{2}^{F}}{m_{1}^{F}+m_{2}^{F}},
\quad
\bar{\theta}_{F}:=\pi\dfrac{2g_{F}}{l}\cdot \theta_{F},
\quad
\epsilon_{F}:=\dfrac{1}{\pi^{2}}\left(\dfrac{l}{2g_{F}}\right)^{2}\dfrac{1}{m_{1}^{F}+m_{2}^{F}}.
\end{equation}
The subscript $F$ of these notations is used to indicate that they are determined by $\Gamma^{F}(M)$.
By the definition of $\delta_{F}$, we have $-1<\delta_{F}<1$ and $\epsilon_{F}>0$ (see \eqref{eqn:epF_range} below).
It follows from Theorem \ref{thm_TeTh95_thm1.6.b}, (i) that $0<\bar{\theta}_{F}<\pi$ holds.
\end{notation}

\begin{thm}\label{thm:equifocal_HA2_formula}
Let $N$ be a simply-connected irreducible symmetric space of compact type.
Let $M$ be an equifocal hypersurface of $N$
and $\bs{\xi}$ a unit normal vector field on $M$.
Then we have the following assertions.
\begin{enumerate}
\item[(i)] $\displaystyle{H=\dfrac{\pi g_{F}(m_{1}^{F}+m_{2}^{F})}{l\cdot \sin\bar{\theta}_{F}}(\delta_{F}+\cos\bar{\theta}_{F})}$.
\item[(ii)] The following equality holds:
\begin{equation}\label{eqn:gLTform_A2_Ric}
\|A\|^{2}
=\dfrac{1}{2\epsilon_{F}\cdot\sin^{2}\bar{\theta}_{F}}(1+\delta_{F}\cos\bar{\theta}_{F})-\kappa.
\end{equation}
\end{enumerate}
\end{thm}

The proof of Theorem \ref{thm:equifocal_HA2_formula}
is given by reducing an equifocal hypersurface
or an equifocal hypersurface
to an infinite-dimensional isoparametric hypersurface in a certain separable Hilbert space.
As preparation for this reduction, in Subsection \ref{sec:infty_isopara},
we review basics on infinite-dimensional isoparametric submanifolds
and explain the relationship between equifocal submanifolds and infinite-dimensional isoparametric submanifolds.
The proof of Theorem \ref{thm:equifocal_HA2_formula} is given in Subsection \ref{sec:proof_LTtypeformula}.

The following corollaries follows immediately from Theorem \ref{thm:equifocal_HA2_formula}.

\begin{cor}\label{cor:equifocal_minimal}
An equifocal hypersurface $M$ with tangential focal data $\Gamma^{F}(M)$
is minimal if and only in the following condition holds:
\begin{equation}
\delta_{F}+\cos\bar{\theta}_{F}=0.
\end{equation}
\end{cor}

By a direct calculation,
we also obtain
the following equalities as a generalization of the formula obtained in \cite[Lemma 4.11]{LiuTe20}.

\begin{cor}
For any constant $c$, the following equality holds:
\begin{multline}
\|A\|^{2}-\dfrac{c}{m_{1}^{F}+m^{2}_{F}}H^{2}
=\\
\dfrac{1}{4\epsilon_{F}\sin^{2}\bar{\theta}_{F}}\{(2-c-c^{2}\delta_{F})+2\delta_{F}(1-c)\cos\bar{\theta}_{F}+(c-4\kappa\epsilon_{F})\sin^{2}\bar{\theta}_{F}\}.
\end{multline}
In particular,
in the case of $\delta_{F}=0$,
we have a simpler form:
\begin{equation}\label{eqn:A-2mmH_ep}
\|A\|^{2}-\dfrac{2}{m_{1}^{F}+m_{2}^{F}}H^{2}\equiv\dfrac{1-2\kappa\epsilon_{F}}{2\epsilon_{F}}.
\end{equation}
\end{cor}

\subsection{Relation between equifocal submanifolds and infinite-dimensional isoparametric submanifolds}\label{sec:infty_isopara}

\subsubsection{Proper Fredholm submanifolds}
We begin by recalling the notion of a proper Fredholm submanifold introduced by Terng \cite{Te89}.

\begin{dfn}\label{dfn:dfn_PF}
Let $(V,\INN{\cdot}{\cdot})$ be a separable Hilbert space.
An immersed Hilbert submanifold $\widetilde{M}$ in $V$ with finite codimension
is called a \textit{proper Fredholm submanifold},
if it satisfies the following conditions:
\begin{enumerate}
\item[(i)] The normal exponential map
$\widetilde{\exp}^{\perp}:T^{\perp}\widetilde{M}\to V$
of $\widetilde{M}$ is a Fredholm map.
\item[(ii)] For each $r>0$, the restriction of $\widetilde{\exp}^{\perp}$ to the normal disk bundle of radius $r$ is a proper map.
\end{enumerate}
\end{dfn}

In Definition \ref{dfn:dfn_PF},
the condition (i) is known to be equivalent to the assertion that
the shape operator of $\widetilde{M}$ in each normal direction is a compact operator.
In particular, since each shape operator is a self-adjoint compact operator,
its non-zero eigenvalues form an at most countable set, each with finite multiplicity,
and their only possible accumulation point is $0$.
On the other hand, the condition (ii) ensures that
the squared distance function determined by a point outside $\widetilde{M}$ satisfies the Palais-Smale condition.

\subsubsection{Infinite-dimensional isoparametric submanifolds}

Next, we recall the notion of an infinite-dimensional isoparametric submanifold,
which is defined as a subclass of proper Fredholm submanifolds.

\begin{dfn}\label{dfn:dfn_isopara_infty}
A proper Fredholm submanifold $\widetilde{M}$ in $V$ is called
an \textit{isoparametric submanifold} if it satisfies the following conditions:
\begin{enumerate}
\item[(i)] The normal holonomy group of $\widetilde{M}$ is trivial.
\item[(ii)] For any parallel unit normal vector field $\bs{v}$ on $\widetilde{M}$
and any two points $u,u'\in\widetilde{M}$,
the shape operators $\widetilde{A}_{\bs{v}_{u}}$ and $\widetilde{A}_{\bs{v}_{u'}}$ are orthogonally equivalent.
Namely, there exists a linear isometry
$\varphi:T_{u}\widetilde{M}\to T_{u'}\widetilde{M}$
satisfying
$
\varphi\circ \widetilde{A}_{\bs{v}_{u}}
=
\widetilde{A}_{\bs{v}_{u'}}\circ \varphi$.
\end{enumerate}
\end{dfn}

\begin{rem}
The condition (i) in Definition~\ref{dfn:dfn_isopara_infty} means that every normal vector at any point $u\in\widetilde{M}$ extends uniquely to a global parallel normal vector field on $\widetilde{M}$ with respect to the normal connection.
Therefore, the normal spaces at different points are uniquely identified by parallel translations with respect to the normal connection.
Under this identification, the condition (ii) means that, for every parallel unit normal vector field,
the spectrum of the corresponding shape operator coincides with considering the multiplicities.
\end{rem}

Let $\widetilde{M}$ be an isoparametric submanifold in $V$. 
Fix an arbitrary point $u\in\widetilde{M}$.
For a linear form $\lambda\in (T^{\perp}_{u}\widetilde{M})^{*}$,
we define the closed subspace
$\widetilde{E}_{\lambda}^{u}$ of $T_{u}\widetilde{M}$ by
\begin{equation}
\widetilde{E}_{\lambda}^{u}
:=
\bigcap_{\bs{v}\in T^{\perp}_{u}\widetilde{M}}
\operatorname{Ker}
(\widetilde{A}_{\bs{v}}
-\lambda(\bs{v})\id).
\end{equation}
By the condition (i) in Definition \ref{dfn:dfn_isopara_infty} and the Ricci equation,
the family of shape operators $\{\widetilde{A}_{\bs{v}}\mid \bs{v}\in T^{\perp}_{u}\widetilde{M}\}$
consists of mutually commuting self-adjoint compact operators.
By Palais-Terng \cite[Proposition 7.2.5]{PTe88},
there exists a common eigenspace decomposition
with respect to these shape operators.
Namely, there exist an at most countable index set $I$ and non-zero linear forms
$\lambda_{i}^{u}
\in (T^{\perp}_{u}\widetilde{M})^{*}$
($i\in I$)
such that $\widetilde{E}_{\lambda_{i}}^{u}\neq\{0\}$ and
\begin{equation}\label{eqn:TtildeM_decomp}
T_{u}\widetilde{M}
=
\overline{
\widetilde{E}_{0}^{u}
\oplus
\bigoplus_{i\in I}
\widetilde{E}_{\lambda_{i}}^{u}
}.
\end{equation}
Here, $\widetilde{E}_{0}^{u}$ is the simultaneous zero eigenspace given by
\begin{equation}
\widetilde{E}_{0}^{u}
=
\bigcap_{\bs{v}\in T^{\perp}_{u}\widetilde{M}}
\operatorname{Ker}\widetilde{A}_{\bs{v}}.
\end{equation}
We note that $\widetilde{E}_{\lambda_{i}}^{u}$ is of finite-dimensional for each $i\in I$.
Set $\widetilde{E}_{i}^{u}:=\widetilde{E}_{\lambda_{i}}^{u}$ for $i\in I$.
According to the condition (ii)
of Definition \ref{dfn:dfn_isopara_infty},
for each $i\in I\cup\{0\}$,
we may assume that the assignment
\begin{equation}
\widetilde{E}_{i}:u\mapsto \widetilde{E}_{i}^{u}
\end{equation}
defines a smooth distribution on $\widetilde{M}$.
These distributions are called the \textit{curvature distributions} of $\widetilde{M}$.
Since the normal holonomy of $T^{\perp}\widetilde{M}$ is trivial, the holonomy group of the dual connection
induced on the dual bundle $(T^{\perp}\widetilde{M})^{*}$ is also trivial.
For each $i\in I$,
we can show that
the section
$\lambda_{i}$ of $(T^{\perp}\widetilde{M})^{*}$ determined by $(\lambda_{i})_{u}:=\lambda_{i}^{u}$
is parallel with respect to this dual connection.
This section $\lambda_{i}$ is called a \textit{principal curvature} of $\widetilde{M}$.
The rank of $\widetilde{E}_{i}$ is called the \textit{multiplicity} of $\lambda_{i}$.
Since the normal bundle $T^{\perp}\widetilde{M}$ has finite rank,
the Riesz representation theorem implies that,
for each $i\in I$, there exists a unique normal vector field $\bs{n}_{i}$ satisfying
\begin{equation}
(\lambda_{i})_{u}(\bs{v})
=
\INN{(\bs{n}_{i})_{u}}{\bs{v}}
\quad
(u\in\widetilde{M},\ \bs{v}\in T^{\perp}_{u}\widetilde{M}).
\end{equation}
This normal vector field $\bs{n}_{i}$ is smooth and parallel, and is called the \textit{curvature normal}.

The normal exponential map
$\widetilde{\exp}^{\perp}:T^{\perp}\widetilde{M}\to V$
of $\widetilde{M}$ is given by
\begin{equation}
\widetilde{\exp}^{\perp}(\bs{v})=u+\bs{v}
\quad
(u\in \widetilde{M},\ \bs{v}\in T^{\perp}_{u}\widetilde{M}).
\end{equation}
We write $\tilde{\psi}:T^{\perp}\widetilde{M}\to \widetilde{M}$
as the natural projection.
For $\bs{v}\in T^{\perp}_{u}\widetilde{M}$,
if
$d\tilde{\psi}_{\bs{v}}(\operatorname{Ker}(d\widetilde{\exp}_{\bs{v}}))\neq\{0\}$
holds, then
$\bs{v}$ is called a \textit{focal normal} of $\widetilde{M}$ at $u$.
In this case, $\dim_{\mathbb{R}}
d\tilde{\psi}_{\bs{v}}(\operatorname{Ker}(d\widetilde{\exp}_{\bs{v}}))$
and $\|\bs{v}\|$ are called the \textit{multiplicity} and the
\textit{focal radius} of along the unit-speed normal geodesic $\gamma_{\bs{v}/\|\bs{v}\|}$, respectively.
We denote by $\mathcal{F}(\widetilde{M},u)$ the set of all focal normals
of $\widetilde{M}$ at $u$.
Let $\bs{v}$ be a parallel normal vector field on $\widetilde{M}$.
We define
$\tilde{\eta}_{\bs{v}}:\widetilde{M}\to V$
by $\tilde{\eta}_{\bs{v}}(u):=\widetilde{\exp}^{\perp}(\bs{v}_{u})$ ($u\in\widetilde{M}$).
This map is called the endpoint map of $\widetilde{M}$ for $\bs{v}$.
According to (ii) of Definition \ref{dfn:dfn_isopara_infty},
$\tilde{\eta}_{\bs{v}}$ is of constant rank.
Set $\widetilde{M}_{\bs{v}}:=\tilde{\eta}_{\bs{v}}(M)$.
If $\tilde{\eta}_{\bs{v}}$ is of full rank, then $\widetilde{M}$
is called a \emph{parallel submanifold of $\widetilde{M}$
for $\bs{v}$};
otherwise, $\widetilde{M}_{\bs{v}}$ is a \emph{focal submanifold of $\widetilde{M}$
for $\bs{v}$}.
When $\widetilde{M}_{\bs{v}}$ is a parallel submanifold of $\widetilde{M}$,
the map $\tilde{\eta}_{\bs{v}}$ is a diffeomorphism from $\widetilde{M}$
onto $\widetilde{M}_{\bs{v}}$.
Let $i\in I$
and $l_{i}^{u}:=(\lambda_{i})_{u}^{-1}(1)$ denote the affine hyperplane in
$T^{\perp}_{u}\widetilde{M}$.
Then it is shown that
\begin{equation}
\mathcal{F}(\widetilde{M},u)
=
\bigcup_{i\in I}l_{i}^{u}.
\end{equation}
The affine hyperplane $l_{i}^{u}$ is called a \textit{focal hyperplane} of $\widetilde{M}$ at $u$.

\subsubsection{Regularizable isoparametric submanifolds}

In general, the shape operators of an isoparametric submanifold
is not necessarily of trace class.
For this reason, we recall the notion of a regularizable submanifold in the sense of Heintze-Liu-Olmos \cite{HLO},
which are defined as a subclass of proper Fredholm submanifolds.

\begin{dfn}[{\cite{HLO}}]\label{dfn:A_regularizable}
Let $W$ be a separable Hilbert space.
A self-adjoint compact operator $\widetilde{A}:W\to W$
is said to be \textit{regularizable},
if it satisfies the following two conditions:
\begin{enumerate}
\item[(i)] $\tr\widetilde{A}^{2}<\infty$.
\item[(ii)] We write
$\lambda^{-}_{1}\leq \lambda^{-}_{2} \leq \dotsb < 0 < \dotsb \leq \lambda_{2}^{+}\leq \lambda_{1}^{+}$
as the eigenvalues of $A$ repeated with multiplicities.
Then the series
\begin{equation}\label{eqn:trr_tildeA_converge}
\sum_{k=1}^{\infty}
\left(\lambda^{+}_{k}+\lambda^{-}_{k}\right),
\end{equation}
where we regard $\lambda_{k}^{+}$ or $\lambda_{k}^{-}$ as $0$
if there are less than $k$ positive or negative eigenvalues, respectively.
\end{enumerate}
The series \eqref{eqn:trr_tildeA_converge} is called the \textit{regularized trace} of $\widetilde{A}$ and is denoted by $\operatorname{tr}_{r}\widetilde{A}$.
\end{dfn}

\begin{dfn}[{\cite{HLO}}]\label{dfn:regularizable_submanifold}
Let $(V,\INN{\cdot}{\cdot})$ be a separable Hilbert space.
A proper Fredholm submanifold $\widetilde{M}$ in $V$
is \textit{regularizable},
if, for any point $u\in\widetilde{M}$ and any normal vector $\bs{v}\in T^{\perp}_{u}\widetilde{M}$,
the shape operator $\widetilde{A}_{\bs{v}}$ is regularizable.
In this case, $\tr_{r}\widetilde{A}_{\bs{v}}$ is called the \textit{regularized mean curvature} of $\widetilde{M}$ in the direction
$\bs{v}$. Assume further that, for each $u\in\widetilde{M}$,
the function on $T^{\perp}_{u}\widetilde{M}$ defined by
$\bs{v}
\to \tr_{r}\widetilde{A}_{\bs{v}}$ is linear.
Then there exists a unique normal vector
$\widetilde{\bs{H}}_{u}\in T^{\perp}_{u}\widetilde{M}$ such that
\begin{equation}
\INN{\widetilde{\bs{H}}_{u}}{\bs{v}}
=
\mathrm{tr}_{r}\widetilde{A}_{\bs{v}}
\quad
(\bs{v}\in T^{\perp}_{u}\widetilde{M}).
\end{equation}
This vector $\widetilde{\bs{H}}_{u}$ is called the
\textit{regularized mean curvature vector} at $u\in\widetilde{M}$.
If $\widetilde{\bs{H}}_{u}$ is defined at every point
$u\in\widetilde{M}$, then the section
$u\to \widetilde{\bs{H}}_{u}$ of
$T^{\perp}\widetilde{M}$ is called the \textit{regularized mean curvature vector field} and is denoted by
$\widetilde{\bs{H}}$.
\end{dfn}

Let $\widetilde{M}$ be a regularizable isoparametric submanifold in $V$.
For any parallel unit normal vector field $\bs{v}$, the regularized trace
$\mathrm{tr}_{r}\widetilde{A}_{\bs{v}_{u}}$ is independent of the choice of $u\in\widetilde{M}$.
Hence, if the regularized mean curvature vector field $\widetilde{\bs{H}}$ is well-defined,
then $\widetilde{\bs{H}}$ is parallel.

\subsubsection{Relation between equifocal submanifolds and infinite-dimensional isoparametric submanifolds}

Let $N=G/K$ be a simply-connected symmetric space of compact type.
Let $\pi:G\to G/K$ be the natural projection.
Let $H^{0}([0,1],\mathfrak{g})$
be the separable Hilbert space which consists
of $L^{2}$-integrable paths $u:[0,1]\to\mathfrak{g}$
and $H^{1}([0,1],G)$
be the Hilbert Lie group which consists of $H^{1}$-paths
$g:[0,1]\to G$.
For a subset $U$ of $G\times G$,
we put
\begin{equation}
 P(G,U):=\{g\in H^{1}([0,1],G) \mid
((g(0),g(1)))\in U\}.
\end{equation}
We note that $P(G,\{e\}\times G)$
is a Hilbert subgroup of $H^{1}([0,1],G)$,
where $e$ denotes the identity element of $G$.
The \textit{parallel transport map}
$\phi:H^{0}([0,1],\mathfrak{g})\to G$
is defined by
\begin{equation}
 \phi(u):=E_{u}(1)\quad
(u\in H^{0}([0,1],\mathfrak{g})),
\end{equation}
where $E_{u}$ denotes the unique solution
of the linear ODE $E_{u}^{-1}E_{u}'=u$ with $E_{u}(0)=e$ (cf.~\cite[Section 4]{TeTh95}):
It follows from \cite[4.5 Theorem, (4)]{TeTh95}
that $\Phi:=\pi\circ\phi:H^{0}([0,1],\mathfrak{g})\to G/K$
is a Riemannian submersion with minimal fibers.

Let $M$ be an immersed compact submanifold in $G/K$.
We set $\widetilde{M}:=\Phi^{-1}(M)$ and call it the lift of $M$ by $\Phi$.
By using \cite[Lemma 5.8]{TeTh95}
it is shown that $\widetilde{M}$ is a proper Fredholm submanifold in $H^{0}([0,1],\mathfrak{g})$.
Let $x\in M$ and $u\in\widetilde{M}$ with $x=\Phi(u)$.
Then $d\Phi_{u}$ gives a linear isometry from the horizontal space of
$\Phi$ at $u$ onto $T_{x}(G/K)$,
which induces a linear isometry between
$T^{\perp}_{u}\widetilde{M}$ and $T^{\perp}_{x}M$.

\begin{lem}[{\cite[Theorems 1.9 and 1.10]{TeTh95}}]\label{lem:TT_MtildeM}
The following assertions hold:
\begin{enumerate}
\item[(i)] $M$ is an equifocal submanifold if and only if each connected component
of $\widetilde{M}$ is an isoparametric submanifold.
\item[(ii)] Let $x\in M$, and let $u\in\widetilde{M}$ satisfy $\Phi(u)=x$.
If we identify $T^{\perp}_{u}\widetilde{M}$ with $T^{\perp}_{x}M$ by $d\Phi_{u}$,
then the sets of focal normal vectors of $M$ and $\widetilde{M}$
coincide with considering the multiplicities.
Namely,
\begin{equation}
d\Phi_{u}
(\mathcal{F}(\widetilde{M},u))
=
\mathcal{F}(M,x)
\end{equation}
holds.
\end{enumerate}
\end{lem}

\begin{lem}[{\cite[Lemma 5.2]{HLO}}]\label{lem:HLO_MtildeM}
Let $M$ be an immersed compact submanifold in $G/K$.
Assume that $M$ is an equifocal submanifold with mean curvature vector field $\bs{H}$.
Then $\widetilde{M}:=\Phi^{-1}(M)$ is a regularizable isoparametric submanifold
which admits the regularized mean curvature vector field $\widetilde{\bs{H}}$.
Furthermore, $\widetilde{\bs{H}}$ coincides with the horizontal lift of $\bs{H}$
with respect to $\Phi$.
\end{lem}

Let $x=\Phi(u)$.
For $\bs{\xi}\in T^{\perp}_{x}M$, let
$\bs{\xi}^{h}_{u}\in T^{\perp}_{u}\widetilde{M}$
be the horizontal lift of $\bs{\xi}$ with respect to $\Phi$.
Also, for $X\in T_{x}M$, let $X^{h}_{u}$ be the horizontal lift of $X$.
Then
\begin{equation}\label{eqn:dPhiF_F}
d\Phi_{u}
(\widetilde{A}_{\bs{\xi}^{h}_{u}}X^{h}_{u})
=
A_{\bs{\xi}}X
\end{equation}
holds.
From the proof of \cite[Lemma 5.2]{HLO}, the following equality is known:
\begin{equation}\label{eqn:H_trrA}
\INN{\bs{H}_{u}}{\bs{\xi}}
=
\tr_{r}\widetilde{A}_{\bs{\xi}^{h}_{u}}.
\end{equation}

\subsection{Proof of Theorem \ref{thm:equifocal_HA2_formula}}\label{sec:proof_LTtypeformula}

Let $N=G/K$ be a simply-connected irreducible symmetric space of compact type.
Let $M$ be an equifocal hypersurface in $N$
with tangential focal data
$\Gamma^{F}(M)
=
\{\theta_{F},l/2g_{F},\{m^{F}_{1},m^{F}_{2}\};2g_{F}\}$.
Let $\Phi:H^{0}([0,1],\mathfrak{g})\to G/K$
be as in the previous subsection
and $\widetilde{M}:=\Phi^{-1}(M)$ denote the lift of $M$ by $\Phi$.
It follows from Lemmas \ref{lem:TT_MtildeM} and \ref{lem:HLO_MtildeM}
that $\widetilde{M}$ is a regularizable isoparametric hypersurface in $H^{0}([0,1],\mathfrak{g})$.

\begin{proof}[Proof of Theorem \ref{thm:equifocal_HA2_formula}, (i)]
Let $x\in M$ be an arbitrary point.
Let $\bs{\xi}\in T^{\perp}_{x}M$ be a unit normal vector.
Let $u\in\widetilde{M}$ with $\Phi(u)=x$.
We denote by $\bs{\xi}_{u}^{h}\in T^{\perp}_{u}\widetilde{M}$
the horizontal lift of $\bs{\xi}$ with respect to $\Phi$ at $u$.
Let $X\in T_{x}M$ be an arbitrary non-zero tangent vector
and $J_{X}$ denote the Jacobi field along the normal geodesic
$\gamma_{\bs{\xi}}$ satisfying
$J_{X}(0)=X$ and
$J'_{X}(0)=-A_{\bs{\xi}}X$.
We write $X^{h}_{u}\in T_{u}\widetilde{M}$ as
the horizontal lift of $X$ with respect to $\Phi$ at $u$.
Then
$\tilde{\gamma}_{\bs{\xi}_{u}^{h}}(s):=u+s\bs{\xi}_{u}^{h}$
gives the normal geodesic of $\widetilde{M}$ through $u$ in the direction $\bs{\xi}_{u}^{h}$.
It is shown that $\gamma_{\bs{\xi}}(s)=\Phi(\tilde{\gamma}_{\bs{\xi}_{u}^{h}}(s))$ holds.
Let $\widetilde{J}_{X^{h}_{u}}$ be the Jacobi field along $\tilde{\gamma}_{\bs{\xi}_{u}^{h}}$ satisfying
$\widetilde{J}_{X^{h}_{u}}(0)=X^{h}_{u}$
and $\widetilde{J}'_{X^{h}_{u}}(0)=-\widetilde{A}_{\bs{\xi}_{u}^{h}}X^{h}_{u}$.
Since $H^{0}([0,1],\mathfrak{g})$ is flat, we get
$\widetilde{J}_{X^{h}_{u}}(s)
=X^{h}_{u}-s\widetilde{A}_{\bs{\xi}_{u}^{h}}X^{h}_{u}$,
where $T_{\tilde{\gamma}_{\bs{\xi}_{u}^{h}}(s)}\widetilde{M}$ is identified with $T_{u}\widetilde{M}$.
Thus, from \eqref{eqn:dPhiF_F}, we have
\begin{equation}
d\eta_{s\cdot\bs{\xi}}(X)=
J_{X}(s)
=d\Phi_{\tilde{\gamma}_{\bs{\xi}_{u}^{h}}(s)}(\widetilde{J}_{X^{h}_{u}}(s))
=d\Phi_{\tilde{\gamma}_{\bs{\xi}_{u}^{h}}(s)}
(X^{h}_{u}-s\widetilde{A}_{\bs{\xi}_{u}^{h}}X^{h}_{u}).
\end{equation}
Hence, a non-zero real number $\lambda$ is an eigenvalue of
$\widetilde{A}_{\bs{\xi}_{u}^{h}}$ if and only if
$\lambda^{-1}\bs{\xi}$ is a focal normal vector of $M$ at $x$.
By using the tangential focal data $\Gamma^{F}(M)$,
the non-zero eigenvalues of $\widetilde{A}_{\bs{\xi}_{u}^{h}}$
are given as follows:
\begin{equation}\label{eqn:tildeA_nzeigen}
\bigcup_{j=1}^{2g_{F}}\{t_{j,k}^{-1}\mid k\in\mathbb{Z}\},\quad
\text{where $t_{j,k}:=\theta_{F}+(j-1)\dfrac{l}{2g_{F}}+k\cdot l$}.
\end{equation}
We put $V_{0}:=\Ker \widetilde{A}_{\bs{\xi}_{u}^{h}}$.
For each $j\in\{1,\ldots,2g_{F}\}$, 
we define $V_{j}$ of $T_{u}\widetilde{M}$ by
\begin{equation}
V_{j}:=\overline{\bigoplus_{k\in\mathbb{Z}}
\Ker(\widetilde{A}_{\bs{\xi}_{x}^{h}}-\dfrac{1}{t_{j,k}}\id)}.
\end{equation}
Then we have
\begin{equation}
T_{u}\widetilde{M}=V_{0}\oplus \bigoplus_{j=1}^{2g_{F}}V_{j}.
\end{equation}
For each $i\in\{0\}\cup\{1,\dotsc,2g_{F}\}$,
we write $\widetilde{A}_{i}$ as the restriction of
$\widetilde{A}_{\bs{\xi}_{u}^{h}}$ to $V_{i}$.
Then we have
\begin{equation}
\widetilde{A}_{\bs{\xi}_{x}^{h}}
=\widetilde{A}_{0}\oplus \bigoplus_{j=1}^{2g_{F}}\widetilde{A}_{j}.
\end{equation}
Since $\widetilde{A}_{\bs{\xi}_{u}^{h}}$ is regularizable, its restriction
$\widetilde{A}_{i}$ to $V_{i}$ is also regularizable for each
$i\in\{0\}\cup\{1,\dotsc,2g_{F}\}$.
It follows from \cite[Lemma 4.4]{HLO}
that the following equality holds:
\begin{equation}\label{eqn:trrA_trrAj}
\tr_{r}\widetilde{A}_{\bs{\xi}_{u}^{h}}
=\sum_{j=1}^{2g_{F}}\tr_{r}\widetilde{A}_{j}.
\end{equation}
In what follows,
we prove the following holds for each $j\in\{1,\dotsc,2g_{F}\}$:
\begin{equation}\label{eqn:trr_Aj}
\tr_{r}\widetilde{A}_{j}=m^{F}_{j (\mathrm{mod}\,2)}\dfrac{\pi }{l}\cot(\pi z_{j}),\quad
\text{where $z_{j}:=\dfrac{\theta_{F}}{l}+(j-1)\dfrac{1}{2g_{F}}$}.
\end{equation}
The non-zero eigenvalues of $\widetilde{A}_{j}$ are given by
$\{t_{j,k}^{-1}\mid k\in\mathbb{Z}\}$.
For $a\in\{1,\dotsc,m^{F}_{j (\mathrm{mod}\,2)}\}$
and $q\in\mathbb{N}\cup\{0\}$, we define the constants
$\lambda_{j,a+qm^{F}_{j (\mathrm{mod}\,2)}}>0$
and $\mu_{j,a+qm^{F}_{j (\mathrm{mod}\,2)}}<0$ by
\begin{equation}
\lambda_{j,a+qm^{F}_{j (\mathrm{mod}\,2)}}:=\dfrac{1}{t_{j,q}},\quad
\mu_{j,a+qm^{F}_{j (\mathrm{mod}\,2)}}:=\dfrac{1}{t_{j,-q-1}}\\.
\end{equation}
Then we have
\begin{equation}
\tr_{r}\widetilde{A}_{j}
=
\sum_{k=1}^{\infty}(\lambda_{j,k}+\mu_{j,k})\\
=
m^{F}_{j\,(\mathrm{mod}\,2)}
\lim_{s\to\infty}
\sum_{q=-s}^{s}
\dfrac{1}{t_{j,q}} \label{eqn:trr_Aj_1}\\
=
m^{F}_{j\,(\mathrm{mod}\,2)}
\dfrac{\pi}{l}\cot(\pi z_{j}).
\end{equation}
Here, in the last equality, we have used Mittag-Leffler expansion:
\begin{equation}
\lim_{s\to\infty}\sum_{q=-s}^{s}\dfrac{1}{z_{j}+s}=\pi\cot(\pi z_{j}).
\end{equation}
Hence \eqref{eqn:trr_Aj} holds.
By combining \eqref{eqn:trrA_trrAj} with \eqref{eqn:trr_Aj}, we obtain
\begin{equation}\label{eqn:trr_Aj_2}
\tr_{r}\widetilde{A}_{\bs{\xi}_{u}^{h}}
=\sum_{j=1}^{2g_{F}}m^{F}_{j (\mathrm{mod}\,2)}\dfrac{\pi }{l}\cot(\pi z_{j})
=\dfrac{\pi }{l}\left\{m_{1}^{F}
\sum_{j=1}^{g_{F}}\cot(\pi z_{2j-1})
+m_{2}^{F}\sum_{j=1}^{g_{F}}\cot(\pi z_{2j})
\right\}.
\end{equation}
Furthermore, by using
\begin{equation}
\sum_{j=1}^{g_{F}}\cot(\pi z_{2j-1})=g_{F}\cot\dfrac{\bar{\theta}_{F}}{2},\quad
\sum_{j=1}^{g_{F}}\cot(\pi z_{2j})
=-g_{F}\tan\dfrac{\bar{\theta}_{F}}{2},
\end{equation}
we get
\begin{equation}
\tr_{r}\widetilde{A}_{\bs{\xi}_{u}^{h}}
=\dfrac{\pi g_{F}}{l}\left\{m_{1}^{F}\cot\dfrac{\bar{\theta}_{F}}{2}
-m_{2}^{F}\tan\dfrac{\bar{\theta}_{F}}{2}\right\}
=\dfrac{\pi g_{F}(m_{1}^{F}+m_{2}^{F})}{l\cdot\sin\bar{\theta}_{F}}(\delta_{F}+\cos\bar{\theta}_{F}).
\end{equation}
Thus, we have completed the proof of 
Theorem \ref{thm:equifocal_HA2_formula}, (i)
because of $\tr_{r}\widetilde{A}_{\bs{\xi}^{h}_{u}}=\tr A_{\bs{\xi}}(=H)$.
\end{proof}

We give some preparations for the proof of Theorem \ref{thm:equifocal_HA2_formula}, (ii).

\begin{lem}\label{lem:DdtA_A2+Rxi}
Let $M$ be an equifocal hypersurface in $N$
with unit normal vector field $\bs{\xi}$.
Let $M_{t\cdot\bs{\xi}}$ be a parallel hypersurface of $M$
and $A(t)$ denote its shape tensor.
Let $\gamma_{\bs{\xi}_{x}}$ be a normal geodesic of $M$
through $x\in M$.
Then the following equality holds:
\begin{equation}
\left.\dfrac{D}{dt}\right|_{t=0}(A(t))_{\gamma_{\bs{\xi}_{x}}'(t)}
=A_{\bs{\xi}_{x}}^{2}+R_{\bs{\xi}_{x}},
\end{equation}
where $R_{\bs{\xi}_{x}}$ denotes $R(\cdot,\bs{\xi}_{x})\bs{\xi}_{x}$.
\end{lem}

\begin{proof}
Let $X\in T_{x}M$ be an arbitrary tangent vector.
Let $c(s)$ be a curve in $M$ satisfying
$c(0)=x$ and
$c'(0)=X$.
Let us consider the geodesic variation defined by
\begin{equation}
F(s,t)
:=
\exp^{\perp}_{c(s)}
(t\bs{\xi}_{c(s)}).
\end{equation}
For each fixed $t$, the curve $s\mapsto F(s,t)$ lies in the parallel hypersurface
$M_{t\cdot \bs{\xi}}$.
For each fixed $s$, the curve $t\mapsto F(s,t)$ is a normal geodesic of $M_{t\cdot \bs{\xi}}$.
Then 
\begin{equation}
J_{X}(t)
:=
\left.
\frac{\partial F}{\partial s}
\right|_{s=0}
\end{equation}
is a Jacobi field along $\gamma_{\bs{\xi}_{x}}$ satisfying
$J_{X}(0)=X$ and
$J_{X}'(0)=-A_{\bs{\xi}_{x}}X$.
Hence $J_{X}(t)$ satisfies the Jacobi equation
$J_{X}''(t)+R_{\gamma_{\bs{\xi}_{x}}'(t)}J_{X}(t)=0$.
On the other hand, since $J_{X}(t)$ is tangent to $M_{t\cdot\bs{\xi}}$, we have
\begin{equation}
J_{X}'(t)
=\nabla_{\gamma_{\bs{\xi}_{x}}'(t)}J_{X}(t)
=\nabla_{J_{X}(t)}\gamma_{\bs{\xi}_{x}}'(t)
=-(A(t))_{\gamma_{\bs{\xi}_{x}}'(t)}J_{X}(t).
\end{equation}
By applying this twice, we obtain
\begin{equation}
J_{X}''(t)
=
-\left(
\dfrac{D}{dt}(A(t))_{\gamma_{\bs{\xi}_{x}}'(t)}
\right)J_{X}(t)
+
\left((A(t))_{\gamma_{\bs{\xi}_{x}}'(t)}\right)^{2}J_{X}(t).
\end{equation}
Hence, by combining this with the Jacobi equation, we obtain
\begin{equation}
\left(
\dfrac{D}{dt}(A(t))_{\gamma_{\bs{\xi}_{x}}'(t)}
\right)J_{X}(t)
=
\left((A(t))_{\gamma_{\bs{\xi}_{x}}'(t)}\right)^{2}J_{X}(t)
+
R_{\gamma_{\bs{\xi}_{x}}'(t)}J_{X}(t).
\end{equation}
By substituting $t=0$ into the both sides, we get the assertion.
\end{proof}

\begin{pro}\label{pro:tildeA_A_Ric}
Let $u\in\widetilde{M}$ with $\Phi(u)=x$.
We denote by $\bs{\xi}_{u}^{h}\in T^{\perp}_{u}\widetilde{M}$
the horizontal lift of $\bs{\xi}_{x}$ with respect to $\Phi$ at $u$.
Then the following equality holds:
\begin{equation}\label{eqn:tildeA_A_Ric}
\tr \widetilde{A}_{\bs{\xi}_{u}^{h}}^{2}
=\|A_{\bs{\xi}_{x}}\|^{2}+\operatorname{Ric}^{N}(\bs{\xi}_{x},\bs{\xi}_{x}).
\end{equation}
\end{pro}

\begin{proof}
Since the parallel hypersurface $M_{t\cdot\bs{\xi}}$
is an equifocal hypersurface in $N$,
its lift $\tilde{M}_{t\cdot\bs{\xi}}:=\Phi^{-1}(M_{t\cdot\bs{\xi}})$
is a regularizable isoparametric hypersurface in $H^{0}([0,1],\mathfrak{g})$.
Let $\widetilde{A}(t)$ denote the shape tensor of $\tilde{M}_{t\cdot\bs{\xi}}$.
It follows from \cite[the proof of Lemma 5.2]{HLO} that
$\tr_{r}(\widetilde{A}(t))_{\tilde{\gamma}_{\tilde{\bs{\xi}}_{u}^{h}}'(t)}
=\tr (A(t))_{\gamma_{\bs{\xi}_{x}}'(t)}$ holds.
By differentiating both sides with respect to $t$ at $t=0$, we obtain
\begin{equation}
\left.\dfrac{d}{dt}\right|_{t=0}
\tr_{r}
(\widetilde{A}(t))_{\tilde{\gamma}_{\tilde{\bs{\xi}}_{u}^{h}}'(t)}
=\left.\dfrac{d}{dt}\right|_{t=0}
\tr (A(t))_{\gamma_{\bs{\xi}_{x}}'(t)}
\end{equation}
Thus, it is sufficient to show the following two equalities:
\begin{equation}\label{eqn:ddt0trrtildeAt}
\left.\dfrac{d}{dt}\right|_{t=0}
\tr_{r}
(\widetilde{A}(t))_{\tilde{\gamma}_{\tilde{\bs{\xi}}_{u}^{h}}'(t)}
=\tr \widetilde{A}_{\bs{\xi}_{u}^{h}}^{2},%
\quad
\left.\dfrac{d}{dt}\right|_{t=0}
\tr (A(t))_{\gamma_{\bs{\xi}_{x}}'(t)}=\tr A_{\bs{\xi}_{x}}^{2}+\operatorname{Ric}^{N}(\bs{\xi}_{x},\bs{\xi}_{x}).
\end{equation}
The hypersurface $\widetilde{M}_{t\cdot \bs{\xi}}:=\Phi^{-1}(M_{t\cdot \bs{\xi}})$ is a parallel hypersurface of $\widetilde{M}$.
More precisely,
for the parallel normal vector field $\widetilde{\bs{\xi}}$ on $\widetilde{M}$ satisfying
$\widetilde{\bs{\xi}}_{u}=\bs{\xi}_{u}^{h}$,
we have $\widetilde{M}_{t\cdot \bs{\xi}}=\tilde{\eta}_{t\cdot\widetilde{\bs{\xi}}}(\widetilde{M})$.
Since $H^{0}([0,1],\mathfrak{g})$ is flat,
we have
\begin{equation}
\dfrac{D}{dt}
(\widetilde{A}(t))_{\tilde{\gamma}_{\bs{\xi}_{u}^{h}}'(t)}
=
(\widetilde{A}(t))_{\tilde{\gamma}_{\bs{\xi}_{u}^{h}}'(t)}^{2},
\end{equation}
in the same method as the proof of Lemma \ref{lem:DdtA_A2+Rxi}.
Hence we have
\begin{equation}
(\widetilde{A}(t))_{\tilde{\gamma}_{\bs{\xi}_{u}^{h}}'(t)}
=
\widetilde{A}_{\bs{\xi}_{u}^{h}}
\left(
\id-t\widetilde{A}_{\bs{\xi}_{u}^{h}}
\right)^{-1}.
\end{equation}
We write
$\mu_{1}\leq \mu_{2}\leq \dotsb <0< \dotsb \leq \lambda_{2}\leq \lambda_{1}$
as the non-zero eigenvalues of
$\widetilde{A}_{\bs{\xi}_{u}^{h}}$
repeated according to their multiplicities.
Then, for sufficiently small $|t|$, the non-zero eigenvalues of
$(\widetilde{A}(t))_{\tilde{\gamma}_{\bs{\xi}_{u}^{h}}'(t)}$ repeated according to their multiplicities
are given by
\begin{equation}
\dfrac{\mu_{1}}{1-t\mu_{1}}
\leq
\dfrac{\mu_{2}}{1-t\mu_{2}}
\leq \dotsb <0< \dotsb \leq
\dfrac{\lambda_{2}}{1-t\lambda_{2}}
\leq
\dfrac{\lambda_{1}}{1-t\lambda_{1}} .
\end{equation}
Hence we have
\begin{equation}
\tr_{r}(\widetilde{A}(t))_{\tilde{\gamma}_{\bs{\xi}_{u}^{h}}'(t)}
=
\sum_{k=1}^{\infty}
\left(
\dfrac{\lambda_{k}}{1-t\lambda_{k}}
+
\dfrac{\mu_{k}}{1-t\mu_{k}}
\right).
\end{equation}
By differentiating both sides with respect to $t$ at $t=0$, we obtain
\begin{equation}
\left.\dfrac{d}{dt}\right|_{t=0}
\tr_{r}
(\widetilde{A}(t))_{\tilde{\gamma}_{\bs{\xi}_{u}^{h}}'(t)}
=
\lim_{t\to 0}
\sum_{k=1}^{\infty}
\left(
\dfrac{\lambda_{k}^{2}}{1-t\lambda_{k}}
+
\dfrac{\mu_{k}^{2}}{1-t\mu_{k}}
\right).
\end{equation}
For sufficiently small $|t|$, we may assume that
\begin{equation}
1-t\lambda_{k}\geq \dfrac{1}{2},
\quad
1-t\mu_{k}\geq \dfrac{1}{2},
\end{equation}
for all $k$.
Then the following inequality holds:
\begin{equation}
\left|
\dfrac{\lambda_{k}^{2}}{1-t\lambda_{k}}
+
\dfrac{\mu_{k}^{2}}{1-t\mu_{k}}
\right|
\leq
2(\lambda_{k}^{2}+\mu_{k}^{2}),
\end{equation}
from which we obtain
\begin{equation}
\sum_{k=1}^{\infty}
\left|
\dfrac{\lambda_{k}^{2}}{1-t\lambda_{k}}
+
\dfrac{\mu_{k}^{2}}{1-t\mu_{k}}
\right|
\leq
2\sum_{k=1}^{\infty}(\lambda_{k}^{2}+\mu_{k}^{2})
=
2\tr \widetilde{A}_{\bs{\xi}_{u}^{h}}^{2}
<\infty.
\end{equation}
Here, in the last inequality, we used the fact that
$\widetilde{A}_{\bs{\xi}_{u}^{h}}$ is regularizable
(see Definition \ref{dfn:A_regularizable}, (1)).
By using the dominated convergence theorem, we obtain
\begin{equation}
\lim_{t\to 0}
\sum_{k=1}^{\infty}
\left(
\dfrac{\lambda_{k}^{2}}{1-t\lambda_{k}}
+
\dfrac{\mu_{k}^{2}}{1-t\mu_{k}}
\right)
=
\sum_{k=1}^{\infty}
\lim_{t\to 0}
\left(
\dfrac{\lambda_{k}^{2}}{1-t\lambda_{k}}
+
\dfrac{\mu_{k}^{2}}{1-t\mu_{k}}
\right)
=
\tr \widetilde{A}_{\bs{\xi}_{u}^{h}}^{2}.
\end{equation}
Hence the first equality in \eqref{eqn:ddt0trrtildeAt} holds.
On the other hand, by using Lemma \ref{lem:DdtA_A2+Rxi}, we have
\begin{equation}
\left.\dfrac{d}{dt}\right|_{t=0}
\tr (A(t))_{\gamma_{\bs{\xi}_{x}}'(t)}
=\tr A_{\bs{\xi}_{x}}^{2}
+
\operatorname{Ric}^{N}(\bs{\xi}_{x},\bs{\xi}_{x}),
\end{equation}
from which the second equality in \eqref{eqn:ddt0trrtildeAt} holds.
Thus, we have completed the proof.
\end{proof}

\begin{pro}\label{pro:tildeA2}
The following equality holds:
\begin{equation}\label{eqn:tildeA2}
\tr \widetilde{A}^{2}_{\bs{\xi}^{h}}
=
\dfrac{1}{2\epsilon_{F}\cdot\sin^{2}\bar{\theta}_{F}}(1+\delta_{F}\cos\bar{\theta}_{F})
\end{equation}
\end{pro}

\begin{proof}
Since the non-zero eigenvalues of $\widetilde{A}_{\tilde{\bs{\xi}}_{x}}$
are given in \eqref{eqn:tildeA_nzeigen},
we get
\begin{equation}\label{eqn:tildeA2_0}
\tr \widetilde{A}^{2}_{\bs{\xi}^{h}}
= \sum_{k\in\mathbb{Z}}\sum_{j=1}^{2g_{F}}\dfrac{1}{t_{j,k}^{2}}
=m_{1}^{F}\sum_{k\in\mathbb{Z}}
\sum_{j=1}^{g_{F}}\dfrac{1}{t_{2j-1,k}^{2}}
+m_{2}^{F}\sum_{k\in\mathbb{Z}}\sum_{j=1}^{g_{F}}\dfrac{1}{t_{2j,k}^{2}}
\end{equation}
Since $j-1+kg_{F}$ runs over all integers exactly once as
$j=1,\dotsc,g_{F}$ and $k\in\mathbb{Z}$,
we have
\begin{equation}
\sum_{k\in\mathbb{Z}}
\sum_{j=1}^{g_{F}}
\dfrac{1}{t_{2j-1,k}^{2}}
=
\sum_{k\in\mathbb{Z}}
\left\{\theta_{F}+k\dfrac{l}{g_{F}}\right\}^{-2} 
=
\dfrac{g_{F}^{2}}{l^{2}}
\sum_{k\in\mathbb{Z}}
\left\{\dfrac{g_{F}}{l}\theta_{F}+k\right\}^{-2}.
\end{equation}
From Theorem \ref{thm_TeTh95_thm1.6.b}, (1),
we get $0<(g_{F}/l)\theta_{F}<1/2$.
In particular, we have $(g_{F}/l)\theta_{F}\notin\mathbb{Z}$.
By using Mittag-Leffler expansion
\begin{equation}
\sum_{k\in\mathbb{Z}}\dfrac{1}{(z+k)^{2}}
=\pi^{2}\csc^{2}(\pi z)
\quad
(z\in\mathbb{C}-\mathbb{Z}),
\end{equation}
we have
\begin{equation}
\sum_{k\in\mathbb{Z}}
\sum_{j=1}^{g_{F}}
\dfrac{1}{t_{2j-1,k}^{2}}
=\dfrac{g_{F}^{2}}{l^{2}}
\cdot\pi^{2}\csc^{2}(\pi\dfrac{g_{F}}{l}\cdot\theta_{F})
=\dfrac{2\pi^{2}g_{F}^{2}}{l^{2}}\dfrac{1}{1-\cos\bar{\theta}_{F}}.\label{eqn:tildeA2_1}
\end{equation}
A similar calculation shows
\begin{equation}\label{eqn:tildeA2_2}
\sum_{k\in\mathbb{Z}}
\sum_{j=1}^{g_{F}}
\dfrac{1}{t_{2j,k}^{2}}
=\dfrac{2\pi^{2}g_{F}^{2}}{l^{2}}\dfrac{1}{1+\cos\bar{\theta}_{F}}.
\end{equation}
By substituting \eqref{eqn:tildeA2_1} and \eqref{eqn:tildeA2_2}
into the right-hand side of \eqref{eqn:tildeA2_0}, we obtain
\eqref{eqn:tildeA2}.
\end{proof}

\begin{proof}[Proof of Theorem \ref{thm:equifocal_HA2_formula}, (ii)]
The equality 
\eqref{eqn:gLTform_A2_Ric}
follows immediately from Proposition \ref{pro:tildeA_A_Ric} and Proposition~\ref{pro:tildeA2}.
\end{proof}

\subsection{Parallel hypersurfaces and Liu-Terng type formulas}

Let $M$ be an equifocal hypersurface in $N$
with tangential focal data
$\Gamma^{F}(M)=\{\theta_{F},l/2g_{F},\{m_{1}^{F},m_{2}^{F}\};2g_{F}\}$.
Let $\bs{\xi}$ be a unit normal vector field on $M$
and $M_{t\cdot\bs{\xi}}:=\eta_{t\cdot\bs{\xi}}(M)$ be a parallel hypersurface of $M$.
Then the tangential focal data of the equifocal hypersurface $M_{t\cdot\bs{\xi}}$
is given as follows:
\begin{equation}\label{eqn:focaldata_parallel}
\Gamma^{F}(M_{t\cdot\bs{\xi}})=\left\{\theta_{F}-t,\dfrac{l}{2g_{F}},\{m_{1}^{F},m_{2}^{F}\};2g_{F}\right\}.
\end{equation}

\begin{notation}\label{dfn:dfn_lambdamuj_t}
Retain the above setting.
We define the constant $\bar{\theta}_{F}^{(t)}$ as follows:
\begin{equation}
\bar{\theta}_{F}^{(t)}:=\pi\dfrac{2g_{F}}{l}\cdot (\theta_{F}-t).
\end{equation}
\end{notation}

\begin{pro}\label{pro:HtAtnorm}
Let $H(t)$ denote the mean curvature of $M_{t\cdot\bs{\xi}}$
and $A(t)$ denote the shape tensor of $M_{t\cdot\bs{\xi}}$.
Then, we have the following equalities:
\begin{align}
H(t)&=\dfrac{\pi g_{F}(m^{F}_{1}+m^{F}_{2})}{l \cdot \sin\bar{\theta}_{F}^{(t)}}
(\delta_{F}+\cos\bar{\theta}_{F}^{(t)}),\label{eqn:H(t)_formula}\\
\|A(t)\|^{2}
&=\dfrac{1}{2\epsilon_{F}\cdot \sin^{2}\bar{\theta}_{F}^{(t)}}
(1+\delta_{F}\cos\bar{\theta}_{F}^{(t)})-\kappa.\label{eqn:A(t)_formula}
\end{align}
\end{pro}

\begin{proof}
The equalities 
\eqref{eqn:H(t)_formula}
and \eqref{eqn:A(t)_formula}
 follow immediately from
Theorem \ref{thm:equifocal_HA2_formula}.
\end{proof}

\begin{cor}\label{cor:minimal_FRA_ca_equi}
Let $M$ be an equifocal hypersurface.
Then, the following assertions hold:
\begin{enumerate}
\item[(i)] There exists a unique minimal equifocal hypersurface
in the family of parallel hypersurfaces of $M$,
which is denoted by $M_{\min}$.
\item[(ii)] $M_{\min}$ is characterized by
the equation given in Corollary \ref{cor:equifocal_minimal}
with the tangential focal data of $M_{\min}$.
\item[(iii)] The squared norm of the shape operator $A^{\min}$ of $M_{\min}$
has the following expression:
\begin{equation}
\|A^{\min}\|^{2}=\dfrac{1-2\kappa\epsilon_{F}}{2\epsilon_{F}}.
\end{equation}
\end{enumerate}
\end{cor}

\begin{proof}
For $t\in(\theta_{F}-(l/2g_{F}), \theta_{F})$,
we have $0<\bar{\theta}_{F}^{(t)}<\pi$.
Hence, there exists a unique 
$t\in(\theta_{F}-(l/2g_{F}), \theta_{F})$
satisfying $\delta_{F}+\cos\bar{\theta}_{F}^{(t)}=0$,
which we write $t=t_{\min}$.
It follows from Corollary \ref{cor:equifocal_minimal}
that $M_{t_{\min}\cdot\bs{\xi}}$ is
a minimal equifocal hypersurface.
The assertion (iii) also follows immediately from \eqref{eqn:A(t)_formula}.
Thus we have the assertion.
\end{proof}

Corollary \ref{cor:minimal_FRA_ca_equi}
shows that the constant $\epsilon_{F}$ can take values only in the range
\begin{equation}\label{eqn:epF_range}
0<\epsilon_{F} \leq \dfrac{1}{2\kappa}.
\end{equation}

\begin{cor}
The minimal equifocal hypersurface $M_{\min}$
is totally geodesic
if and only if
$2\kappa\epsilon_{F}=1$ holds.
\end{cor}

\begin{rem}\label{rem:sphere_ep_g}
Let $S^{n+1}(r)$ be the standard sphere of radius $r$ in $\mathbb{R}^{n+2}$.
Let us consider the case when $M$ is an isoparametric hypersurface in $S^{n+1}(r)$.
Then the Einstein constant $\kappa$ of $S^{n+1}(r)$ is equal to $n/r^{2}$.
We also have $l=2\pi r$.
In addition, 
by $2n=g_{F}(m_{1}^{F}+m_{2}^{F})$ (cf. \cite[(4.9)]{LiuTe20}),
we get $2\kappa\epsilon_{F}=g_{F}^{-1}$,
from which
$2\kappa\epsilon_{F}=1$ is equivalent to $g_{F}=1$.
It is known that any isoparametric hypersurface with $g_{F}=1$
is precisely a hypersphere in $S^{n+1}$.
\end{rem}

\section{Backward mean curvature flows starting from equifocal hypersurfaces}\label{sec:FRA_MCF}

In this section,
we study the mean curvature flows
and the backward mean curvature flows
starting from equifocal hypersurfaces.
We begin with a review of the ODE obtained by reducing the mean curvature equation
using the symmetry of equifocal hypersurfaces due to \cite{Koike11} (see \eqref{eqn:MCF_ODE} below).
Let $N$ be a simply-connected irreducible symmetric space of compact type
and $\kappa$ be the Einstein constant of $N$.
Let $M$ be an equifocal hypersurface in $N$.
Let $\bs{\xi}$ be a unit normal vector field on $M$.
Fix $x\in M$
and let $\gamma_{\bs{\xi}_{x}}$
denote the normal geodesic of $M$
satisfying 
$\gamma_{\bs{\xi}_{x}}(0)=x$
and $\gamma_{\bs{\xi}_{x}}'(0)=\bs{\xi}_{x}$.
Let $\widetilde{C}\subset T^{\perp}_{x}M$
denote the fundamental domain of the affine Coxeter group of $M$ at $x$,
as given in \eqref{eqn:Coxeter_domain}.
Let $\bs{\xi}(t)=\xi(t)\bs{\xi}_{x}$
be the maximal integral curve in $\widetilde{C}$
defined on the interval $t\in(-T_{\min},T_{\max})$
such that it satisfies the following ODE:
 \begin{equation}\label{eqn:MCF_ODE}
\bs{\xi}'(t)=\bs{H}(\xi(t))_{\gamma_{\bs{\xi}_{x}}(\xi(t))},
\end{equation}
where $\bs{H}(\xi(t))$ denotes the mean curvature vector of the parallel
hypersurface $M_{\tilde{\bs{\xi}}(t)}$ and
$\tilde{\bs{\xi}}(t)$ denotes the parallel normal vector field on $M$ satisfying
$\tilde{\bs{\xi}}(t)_{x}=\bs{\xi}(t)$.
Here, we have used the identification of $T_{\gamma_{\bs{\xi}_{x}}(\xi(t))}M_{\tilde{\bs{\xi}}(t)}$
with $T_{x}M$ via the parallel transformation along $\gamma_{\bs{\xi}_{x}}$.
Then, by \cite{Koike11}, $\{M_{\tilde{\bs{\xi}}(t)}\}_{t\in[0,T_{\max})}$ gives a solution
of the mean flow starting from $M$.
In addition, by \cite{LiuRa22},
$\{M_{\tilde{\bs{\xi}}(-t)}\}_{t\in[0,T_{\min})}$
gives a long-time solution of the backward mean curvature flow stating from $M$.
In particular, $T_{\min}=\infty$ holds.

\begin{pro}\label{pro:t-infty_costheta}
The following equality holds:
\begin{equation}
\delta_{F}+\bar{\theta}_{F}^{(\xi(t))}
=\exp(\dfrac{1}{2\epsilon_{F}}\cdot t)(\delta_{F}+\cos\bar{\theta}_{F}).
\end{equation}
\end{pro}

\begin{proof}
Let $H(\xi(t))=\INN{\bs{H}(\xi(t))}{\gamma_{\bs{\xi}_{x}}'(t)}$ denote the mean curvature of
$M_{\tilde{\bs{\xi}}(t)}$.
It follows from \eqref{eqn:MCF_ODE}
that $\xi'(t)=H(\xi(t))$ holds.
By using \eqref{eqn:H(t)_formula},
we have the following equality:
\begin{equation}
H(\xi(t))
=\dfrac{\pi g_{F}(m^{F}_{1}+m^{F}_{2})}{l \cdot \sin\bar{\theta}_{F}^{(\xi(t))}}
(\delta_{F}+\cos\bar{\theta}_{F}^{(\xi(t))}),
\end{equation}
from which we get
\begin{align}
(\delta_{F}+\cos\bar{\theta}_{F}^{(\xi(t))})'
&=-\sin\bar{\theta}_{F}^{(\xi(t))}
\cdot \pi\dfrac{2g_{F}}{l}\cdot(-\xi'(t))\\
&=
\sin\bar{\theta}_{F}^{(\xi(t))}\cdot
\pi\dfrac{2g_{F}}{l}\cdot
\dfrac{\pi g_{F}(m^{F}_{1}+m^{F}_{2})}{l \cdot \sin\bar{\theta}_{F}^{(\xi(t))}}
(\delta_{F}+\cos\bar{\theta}_{F}^{(\xi(t))})\\
&=\dfrac{1}{2\epsilon_{F}}(\delta_{F}+\cos\bar{\theta}_{F}^{(\xi(t))}). \label{eqn:dfcostfx'}
\end{align}
By integrating the both sides of \eqref{eqn:dfcostfx'},
we get the assertion.
\end{proof}

We observe that,
by Proposition \ref{pro:t-infty_costheta},
the following equality holds:
\begin{equation}\label{eqn:cos_t2-inf}
\lim_{t\to -\infty}\cos\bar{\theta}_{F}^{(\xi(t))}=-\delta_{F}.
\end{equation}
This implies that
$M_{\tilde{\bs{\xi}}(t)}$
converges to $M_{\min}$
as $t\to-\infty$.
In what follows,
we study the behavior of $H(\xi(t))$
and $\|A(\xi(t))\|^{2}$ as $t\to -\infty$.

\begin{pro}\label{pro:FRA_Ht_-inf}
The following equality holds:
\begin{equation}
\lim_{t\to-\infty}
H(\xi(t))^{2}\exp(-\dfrac{1}{\epsilon_{F}}\cdot t)
=\dfrac{m_{1}^{F}+m_{2}^{F}}{4\epsilon_{F}(1-\delta_{F}^{2})}
(\delta_{F}+\cos\bar{\theta}_{F})^{2}.
\end{equation}
\end{pro}

\begin{proof}
By using \eqref{eqn:H(t)_formula} and Proposition \ref{pro:t-infty_costheta},
we get
\begin{align}
H(\xi(t))^{2}\exp(-\dfrac{1}{\epsilon_{F}}\cdot t)
&=\dfrac{\pi^{2} g_{F}^{2}(m^{F}_{1}+m^{F}_{2})^{2}}{l^{2} \cdot \sin^{2}\bar{\theta}_{F}^{(\xi(t))}}
(\delta_{F}+\cos\bar{\theta}_{F}^{(\xi(t))})^{2}\exp(-\dfrac{1}{\epsilon_{F}}\cdot t)\\
&=\dfrac{m_{1}^{F}+m_{2}^{F}}{4\epsilon_{F}(1-\cos^{2}\bar{\theta}_{F}^{(\xi(t))})}
(\delta_{F}+\cos\bar{\theta}_{F})^{2}.
\end{align}
Therefore, by taking the limit as $t\to -\infty$ on both sides, we obtain the assertion.
\end{proof}

\begin{pro}\label{pro:FRA_At_-inf}
The following equality holds:
\begin{equation}
\lim_{t\to-\infty}\|A(\xi(t))\|^{2}
=\dfrac{1-2\kappa\epsilon_{F}}{2\epsilon_{F}}.
\end{equation}
\end{pro}

\begin{proof}
It follows from \eqref{eqn:A(t)_formula} that
the following equality holds:
\begin{equation}
\|A(\xi(t))\|^{2}
=
\dfrac{1}{2\epsilon_{F}(1-\cos^{2}\bar{\theta}_{F}^{(\xi(t))})}
(1+\delta_{F}\cos\bar{\theta}_{F}^{(\xi(t))})-\kappa.
\end{equation}
As $t\to -\infty$,
a direct calculation by using
\eqref{eqn:cos_t2-inf} shows the assertion.
\end{proof}

\begin{rem}
Let us consider the case when $M$ is a isoparametric hypersurface in the $(n+1)$-dimensional standard sphere $S^{n+1}(r)$.
It follows from Remark \ref{rem:sphere_ep_g} that $2\kappa\epsilon_{F}=1$ yields $g_{F}=1$.
Furthermore, this implies that $m_{1}^{F}=m_{2}^{F}$ holds.
Thus, $\delta_{F}=0$ holds.
\end{rem}

\begin{pro}\label{pro:eF1_deltaF0}
Let $M$ be an equifocal hypersurface
in $N$.
Then $2\kappa\epsilon_{F}=1$ implies $\delta_{F}=0$.
\end{pro}

\begin{proof}
Assume that $2\kappa\epsilon_{F}=1$.
We get $\{\cos\bar{\theta}_{F}^{(s)}\mid s\in (\theta_{F}-l/(2g_{F}),\theta_{F})\}=(-1,1)$.
By $-1<\delta_{F}<1$, there exist $s_{1},s_{2}\in(\theta_{F}-l/(2g_{F}),\theta_{F})$
satisfying
\begin{equation}\label{eqn:d+cos12}
\delta_{F}+\cos\bar{\theta}_{F}^{(s_{1})}>0,\quad
\delta_{F}+\cos\bar{\theta}_{F}^{(s_{2})}<0,
\end{equation}
respectively.
Then the parallel hypersurfaces
$M_{1}:=M_{s_{1}\cdot\bs{\xi}}$
and
$M_{2}:=M_{s_{2}\cdot\bs{\xi}}$
are non-minimal equifocal hypersurfaces.
For $i=1,2$,
let $\{M_{i}^{(t)}\}_{t\in(-\infty,T_{i})}$
denote the maximal solution
of MCF
starting from $M_{i}$, where $T_{i}$ is a positive real number.
We write $\Gamma^{F}(M_{i}^{(t)})=\{(\theta_{F,i}(t),l/(2g_{F}),\{m_{1}^{F},m_{2}^{F}\};2g_{F})\}$
as the tangential focal data of $M_{i}^{(t)}$.
We put $\bar{\theta}_{F,i}(t):=\pi (2g_{F}/l)\theta_{F,i}(t)$.
Then $M_{i}^{(0)}=M_{i}$ yields $\bar{\theta}_{F,i}(0)=\bar{\theta}_{F}^{(s_{i})}$.
It follows from the uniqueness of the minimal parallel hypersurface of $M$
that
$M_{1}^{(t)}$
and $M_{2}^{(t)}$
converge to $M_{\min}$ as $t\to -\infty$.
Hence we have $\cos(\bar{\theta}_{F,i}(t))\to-\delta_{F}$ as $t\to-\infty$.
Let $A_{i}(t)$ and $H_{i}(t)$ denote
the shape operator and mean curvature of $M_{i}^{(t)}$, respectively.
By using $2\kappa\epsilon_{F}=1$, we get
\begin{equation}
\|A_{i}(t)\|^{2}
=\dfrac{1}{2\epsilon_{F}\cdot\sin^{2}(\bar{\theta}_{F,i}(t))}
\cos(\bar{\theta}_{F,i}(t))
(\delta_{F}+\cos(\bar{\theta}_{F,i}(t))),
\end{equation}
from which we have
\begin{equation}
\dfrac{\|A_{i}(t)\|^{2}}{H_{i}(t)^{2}}
=
\dfrac{2}{m_{1}^{F}+m_{2}^{F}}\cdot\dfrac{\cos(\bar{\theta}_{F,i}(t))}{\delta_{F}+\cos(\bar{\theta}_{F,i}(t))}.
\end{equation}
By using Proposition \ref{pro:t-infty_costheta},
we have
\begin{align}
\lim_{t\to-\infty}\dfrac{\|A_{i}(t)\|^{2}}{H_{i}(t)^{2}}\exp(\dfrac{1}{2\epsilon_{F}}\cdot t)
&=\dfrac{2}{m_{1}^{F}+m_{2}^{F}}\cdot\dfrac{1}{\delta_{F}+\cos\bar{\theta}_{F}^{(s_{i})}}
\lim_{t\to-\infty}\cos(\bar{\theta}_{F,i}(t))\\
&=-\dfrac{2}{m_{1}^{F}+m_{2}^{F}}\cdot\dfrac{\delta_{F}}{\delta_{F}+\cos\bar{\theta}_{F}^{(s_{i})}}.\label{eqn:-d_d+cos}
\end{align}
Suppose that $\delta_{F}\neq 0$.
Then \eqref{eqn:-d_d+cos} implies that
$\delta_{F}+\cos\bar{\theta}_{F}^{(s_{1})}=\delta_{F}+\cos\bar{\theta}_{F}^{(s_{2})}$.
This contradicts \eqref{eqn:d+cos12}.
Hence we have $\delta_{F}=0$.
\end{proof}

\begin{thm}\label{thm:FRA_At_Ht_ratio}
Assume that the initial equifocal hypersurface $M$ is not minimal.
Then we obtain the following assertions:
\begin{enumerate}
\item[(i)] In the case of $2\kappa\epsilon_{F}=1$,
we have
\begin{equation}
\dfrac{\|A(\xi(t))\|^{2}}{H(\xi(t))^{2}}\equiv
\dfrac{2}{m_{1}^{F}+m_{2}^{F}}.
\end{equation}
\item[(ii)] Otherwise,
the following inequality holds:
\begin{equation}\label{eqn:main3}
\dfrac{\|A(\xi(t))\|^{2}}{H(\xi(t))^{2}}
<\dfrac{4\max\{m_{1}^{F},m_{2}^{F}\}}{(m^{F}_{1}+m^{F}_{2})^{2}}
\cdot\dfrac{1}{(\delta_{F}+\cos\bar{\theta}_{F})^{2}}\cdot
\exp(-\dfrac{1}{\epsilon_{F}}\cdot t).
\end{equation}
\end{enumerate}
\end{thm}

\begin{proof}
(i) This assertion is immediately from \eqref{eqn:A-2mmH_ep} and Proposition \ref{pro:eF1_deltaF0}.
(ii) By using Propositions \ref{pro:HtAtnorm} and \ref{pro:t-infty_costheta},
a direct calculation shows
\begin{equation}\label{eqn:ahext_eqn}
\dfrac{\|A(\xi(t))\|^{2}}{H(\xi(t))^{2}}\exp(\dfrac{1}{\epsilon_{F}}\cdot t)
=\dfrac{2}{m_{1}^{F}+m_{2}^{F}}\dfrac{1-2\kappa\epsilon_{F}+\delta_{F}\cos\bar{\theta}_{F}^{(\xi(t))}+2\kappa\epsilon_{F}\cos^{2}\bar{\theta}_{F}^{(\xi(t))}}{(\delta_{F}+\cos\bar{\theta}_{F})^{2}}.
\end{equation}
Here, we consider the following function $f$ of $u=\cos\bar{\theta}_{F}^{(\xi(t))}$:
\begin{equation}
f(u):=1-2\kappa\epsilon_{F}+\delta_{F}u+2\kappa\epsilon_{F}u^{2}
\quad (-1<u<1).
\end{equation}
The axis of this quadratic function $f$ is given by the following equation:
\begin{equation}
u=-\dfrac{\delta_{F}}{4\kappa\epsilon_{F}}.
\end{equation}
If $\delta_{F}\geq 0$, then $f(u)<f(1)=1+\delta_{F}$.
If $\delta_{F}<0$, then $f(u)<f(-1)=1-\delta_{F}$.
Hence, in both cases, we have $f(u)<1+|\delta_{F}|$.
From this, we conclude:
\begin{align}
\dfrac{\|A(\xi(t))\|^{2}}{H(\xi(t))^{2}}\exp(\dfrac{1}{\epsilon_{F}}\cdot t)
&<\dfrac{2(1+|\delta_{F}|)}{(m^{F}_{1}+m^{F}_{2})}
\cdot\dfrac{1}{(\delta_{F}+\cos\bar{\theta}_{F})^{2}}\\
&=\dfrac{4\max\{m_{1}^{F},m_{2}^{F}\}}{(m^{F}_{1}+m^{F}_{2})^{2}}
\cdot\dfrac{1}{(\delta_{F}+\cos\bar{\theta}_{F})^{2}}.
\end{align}
Thus we have completed the proof.
\end{proof}

\begin{thm}\label{thm:BMCF_HA_evaluate}
Assume that the initial equifocal hypersurface $M$
is not minimal.
In the case of $2\kappa\epsilon_{F} \neq 1$,
there exist $t_{0}>0$ and $c_{1},c_{2}>0$ satisfying the following inequalities:
\begin{equation}\label{eqn:A^2/H^2_evaluate}
c_{2}\cdot\exp(-\dfrac{1}{\epsilon_{F}}\cdot t)
\leq 
\dfrac{\|A(\xi(t))\|^{2}}{H(\xi(t))^{2}}
\leq c_{1}\cdot\exp(-\dfrac{1}{\epsilon_{F}}\cdot t)
\quad
(t<-t_{0}).
\end{equation}
\end{thm}

\begin{proof}
By using Propositions
\ref{pro:FRA_At_-inf} and
\ref{pro:FRA_Ht_-inf},
we have
\begin{equation}
\lim_{t\to-\infty}
\dfrac{\|A(\xi(t))\|^{2}}{H(\xi(t))^{2}}\exp(\dfrac{1}{\epsilon_{F}}\cdot t)
=\dfrac{2}{m_{1}^{F}+m_{2}^{F}}
\dfrac{(1-\delta_{F}^{2})(1-2\kappa\epsilon_{F})}
{(\delta_{F}+\cos\bar{\theta}_{F})^{2}}>0
\end{equation}
from which the assertion holds.
\end{proof}

Finally, we give an estimate for the traceless part 
$\mathring{A}(\xi(t))$ of $A(\xi(t))$
along the solution $\{M_{\tilde{\bs{\xi}}(t)}\}_{t\in(-\infty,T_{\max})}$.
We put
\begin{equation}\label{eqn:phixit}
\phi(\xi(t)):=\|\mathring{A}(\xi(t))\|^{2},
\quad
\mathring{A}(\xi(t)):=A(\xi(t))-\dfrac{H(\xi(t))}{\dim M_{\tilde{\bs{\xi}}(t)}}\mathrm{id}.
\end{equation}
It is verified that the following equality holds:
\begin{equation}\label{eqn:mathringA_A_H}
\phi(\xi(t))=\|A(\xi(t))\|^{2}-\dfrac{1}{\dim M}H(\xi(t))^{2}.
\end{equation}
Since $M_{\tilde{\bs{\xi}}(t)}$ converges to $M_{\min}$ as $t\to-\infty$,
we have
\begin{equation}
\lim_{t\to-\infty}\phi(\xi(t))=\dfrac{1-2\kappa\epsilon_{F}}{2\epsilon_{F}}.
\end{equation}
Hence, for any $0<h<1$,
there exists $t_{0}>0$ satisfying the following inequalities:
\begin{equation}\label{eqn:ineq_mathringA}
\dfrac{1-2\kappa\epsilon_{F}}{2\epsilon_{F}}-h
\leq
\phi(\xi(t))
\leq
\dfrac{1-2\kappa\epsilon_{F}}{2\epsilon_{F}}+h
\quad
(t < -t_{0}).
\end{equation}
However, as pointed out by Liu-Terng \cite[Remark 4.19]{LiuTe20},
it follows from \cite{O70} that
there exists a minimal hypersurface $L$ in a sphere which is not isoparametric
and whose shape operator $A^{L}$ satisfies $n-c<\|A^{L}\|^{2}<n+c$
for any sufficiently small $c>0$.
This motivates the following refinement of \eqref{eqn:ineq_mathringA}.

\begin{thm}\label{thm:prop_phi}
Let $N$ be a simply-connected irreducible symmetric space of compact type
and $\kappa$ be the Einstein constant of $N$.
Then, for any $0<h<1$, there exists $c_{0}>0$ such that
$\phi(\xi(t))$ defined in \eqref{eqn:phixit}
satisfies one of the followings:
\begin{enumerate}
\item[(i)] The case of $\delta_{F}=0$:
If $\theta_{F}\in (\theta_{\min}-c_{0},\theta_{\min}+c_{0})$,
we have
\begin{equation}
\dfrac{1-2\kappa\epsilon_{F}}{2\epsilon_{F}}
\leq
\phi(\xi(t))
\leq
\dfrac{1-2\kappa\epsilon_{F}}{2\epsilon_{F}}+h\quad
(t<0).
\end{equation}
\item[(ii)] The case of $\delta_{F}>0$: If $\theta_{F}\in (\theta_{\min}-c_{0},\theta_{\min})$,
we have
\begin{equation}
\dfrac{1-2\kappa\epsilon_{F}}{2\epsilon_{F}}-h
\leq
\phi(\xi(t))
\leq
\dfrac{1-2\kappa\epsilon_{F}}{2\epsilon_{F}}
\quad
(t<0).
\end{equation}
If $\theta_{F}\in (\theta_{\min},\theta_{\min}+c_{0})$,
we have
\begin{equation}
\dfrac{1-2\kappa\epsilon_{F}}{2\epsilon_{F}}\leq
\phi(\xi(t))
\leq
\dfrac{1-2\kappa\epsilon_{F}}{2\epsilon_{F}}+h
\quad
(t<0).
\end{equation}
\item[(iii)] The case of $\delta_{F}<0$:
If $\theta_{F}\in (\theta_{\min}-c_{0},\theta_{\min})$,
we have
\begin{equation}
\dfrac{1-2\kappa\epsilon_{F}}{2\epsilon_{F}}
\leq
\phi(\xi(t))
\leq
\dfrac{1-2\kappa\epsilon_{F}}{2\epsilon_{F}}+h
\quad
(t<0).
\end{equation}
If $\theta_{F}\in (\theta_{\min},\theta_{\min}+c_{0})$,
we have
\begin{equation}
\dfrac{1-2\kappa\epsilon_{F}}{2\epsilon_{F}}-h
\leq
\phi(\xi(t))
\leq
\dfrac{1-2\kappa\epsilon_{F}}{2\epsilon_{F}}
\quad
(t<0).
\end{equation}
\end{enumerate}
\end{thm}

\begin{proof}
For simplicity, we put $u:=\cos\bar{\theta}_{F}^{(\xi(t))}$.
By using 
\eqref{eqn:H(t)_formula},
\eqref{eqn:A(t)_formula} and
\eqref{eqn:mathringA_A_H},
a direct calculation shows
\begin{align}
\phi(\xi(t))-\dfrac{1-2\kappa\epsilon_{F}}{2\epsilon_{F}}
&=\dfrac{1}{2\epsilon_{F}}\cdot
\dfrac{1}{1-u^{2}}(u+\delta_{F})\left\{u-\dfrac{m_{1}^{F}+m_{2}^{F}}{2\dim M}(u+\delta_{F})\right\}\\
&=\dfrac{2\dim M-(m_{1}^{F}+m_{2}^{F})}{4\epsilon_{F} \dim M}
\cdot
\dfrac{1}{1-u^{2}}(u+\delta_{F})\left\{u-\dfrac{(m_{1}^{F}+m_{2}^{F})\delta_{F}}{2\dim M - (m_{1}^{F}+m_{2}^{F})}\right\}.
\end{align}
We define the function $f$ of $u$ by
\begin{equation}
f(u):=\dfrac{1}{1-u^{2}}(u+\delta_{F})\left\{u-\dfrac{(m_{1}^{F}+m_{2}^{F})\delta_{F}}{2\dim M - (m_{1}^{F}+m_{2}^{F})}\right\}
\quad
(-1<u<1).
\end{equation}
(i) When $\delta_{F}=0$, the expression of the function $f$ simplifies
$f(u)=u^{2}/(1-u^{2})\geq 0$.
The equality holds if and only if $u=0$.
Hence we have $\phi(\xi(t))\geq (1-2\kappa\epsilon_{F})/(2\epsilon_{F})$.

When $\delta_{F}\neq 0$,
we obtain
\begin{equation}
f'(-\delta)=-\dfrac{2\dim M}{2\dim M-(m_{1}^{F}+m_{2}^{F})}\cdot\dfrac{\delta_{F}}{1-\delta_{F}^{2}}.
\end{equation}
(ii) $\delta_{F}>0$ yields $f'(-\delta_{F})<0$,
from which the function $f$ is decreasing
in a neighborhood of $u=-\delta_{F}$.
Hence the assertion (ii) holds.
A similar argument shows the assertion (iii) holds.
Therefore, we have completed the proof.
\end{proof}

\appendix

\section{The constancy of $\operatorname{Ric}^{N}(\xi,\xi)$}
\label{sec:const_ricxixi}

We show that the following fact holds
for any (connected) equifocal hypersurface in a reducible symmetric space
of compact type.

\begin{pro}
Let $N$ be a reducible symmetric space of compact type
and $M$ be a connected equifocal hypersurface in $N$.
Then $\operatorname{Ric}^{N}(\bs{\xi},\bs{\xi})$ is constant  on $M$
for a globally defined unit normal vector field $\bs{\xi}$.
\end{pro}

\begin{proof}
Let $x$ and $y$ be arbitrary points of $M$.
Since $M$ is connected, there exists a smooth curve
$c:[0,1]\to M$ such that $c(0)=x$ and $c(1)=y$.
Denote by
$P^{\perp}_{c}|_{[0,1]}:T^{\perp}_{x}M\to T^{\perp}_{y}M$
the parallel transport with respect to the normal connection
$\nabla^\perp$ along $c$.
It follows from
\cite[Theorem 5.1 and Corollary 5.2]{Bruck} that
there exists an isometry $g$ in the identity component
of the isometry group of $N$ such that
$g(x)=y$
and
$(dg)_{x}|_{T^{\perp}_{x}M}=P^{\perp}_{c}|_{[0,1]}$.
In particular, we get $\bs{\xi}_{y}=(dg)_{x}(\bs{\xi}_{x})$.
Since $\operatorname{Ric}^{N}$ is invariant under the ambient isometry $g$,
we obtain
\begin{equation}
\operatorname{Ric}^{N}_{y}(\bs{\xi}_{y},\bs{\xi}_{y})
=\operatorname{Ric}^{N}_{g(x)}((dg)_{x}(\bs{\xi}_{x}),(dg)_{x}(\bs{\xi}_{x}))
=\operatorname{Ric}^{N}_{x}(\bs{\xi}_{x},\bs{\xi}_{x}). 
\end{equation}
By the arbitrariness of $x$ and $y$,
$\operatorname{Ric}^N(\bs{\xi},\bs{\xi})$ is constant on $M$.
\end{proof}

\subsubsection*{Acknowledgements}
The first author would like to express his sincere gratitude to Professor Peter Quast
for the warm hospitality, continuous support, and stimulating discussions
during his stay at the University of Augsburg.

\end{document}